\documentclass{amsart}
\usepackage{amsmath,amssymb,amsfonts,amsthm,calc,color,bm}

\usepackage[latin1]{inputenc}
\usepackage{fullpage}
\usepackage[dvipsnames]{xcolor}
\usepackage{xkeyval}
\usepackage{graphicx}
\usepackage{epstopdf}

\usepackage[german, english]{babel}
\usepackage{hyperref}
\usepackage{enumitem}

\usepackage{xcolor}
\pagecolor{white}

\newcommand{\be}{\begin}
\newcommand{\e}{\end}
\newcommand{\beq}{\begin{equation}}
\newcommand{\eeq}{\end{equation}}

\newtheorem{lm}{Lemma}[section]
\newtheorem{thm}[lm]{Theorem}
\newtheorem*{thm*}{Theorem}
\newtheorem{prop}[lm]{Proposition}
\newtheorem{cor}[lm]{Corollary}

\theoremstyle{definition}
\newtheorem{defn}[lm]{Definition}

\newtheorem{rmk}[lm]{Remark}

\numberwithin{equation}{section}
\newcommand{\comment}[1]{}

\newcommand{\setof}[2]{\left\{ #1\; : \;#2 \right\}}

\newcommand{\Z}{{\mathbb Z}}
\newcommand{\R}{{\mathbb R}}
\newcommand{\C}{{\mathbb C}}

\newcommand{\T}{{\mathbb T}}

\newcommand{\al}{{\alpha}}
\newcommand{\de}{{\delta}}
\newcommand{\eps}{{\varepsilon}}

\newcommand{\lam}{{\lambda}}

\newcommand{\oline}{\overline}

\renewcommand{\l}{\left}
\renewcommand{\r}{\right}

\newcommand{\Hm}[1]{\leavevmode{\marginpar{\tiny%
$\hbox to 0mm{\hspace*{-0.5mm}$\leftarrow$\hss}%
\vcenter{\vrule depth 0.1mm height 0.1mm width \the\marginparwidth}%
\hbox to 0mm{\hss$\rightarrow$\hspace*{-0.5mm}}$\\\relax\raggedright
#1}}}

\newcommand{\ceil}[1]{\ensuremath{\left\lceil #1 \right\rceil}}

\renewcommand{\d}{\mathrm{d}}

 \usepackage{amssymb}
 \usepackage{comment}
 \usepackage{graphics}
\usepackage{bbm}

\newcommand{\ti}{\textnormal{i}}

\begin{document}

\title[Universal Eigenvalue Statistics for Dynamically Defined Matrices]{Universal Eigenvalue Statistics for\\ Dynamically Defined Matrices}

\author[A.~Adhikari]{Arka Adhikari}
\address{Arka Adhikari, Department of Mathematics, Stanford University, 94305 Stanford, USA}
\email{arkaa@stanford.edu}
\author[M.~Lemm]{Marius Lemm}
\address{Marius Lemm, Department of Mathematics, University of T\"ubingen, 72076 T\"ubingen, Germany}
\email{marius.lemm@uni-tuebingen.de}
\date{January 3, 2022}

\begin{abstract}
We consider dynamically defined Hermitian matrices generated from orbits of the doubling map. 
We prove that their spectra fall into the GUE universality class from random matrix theory.  
\end{abstract}

\maketitle

\section{Introduction}
Eugene Wigner formulated the vision that the spectra of complex quantum systems are well described by the eigenvalues of random matrices. His work inspired the Wigner-Dyson-Mehta universality conjecture that the eigenvalues of matrices with independent entires are \textit{universal} in the sense that they only depend on the symmetry type of the matrix---Hermitian (GUE) or real symmetric (GOE). The Wigner-Dyson-Mehta universality conjecture was a driving force of random matrix theory and was finally resolved around 10 years ago in a series of groundbreaking works \cite{ Erdos2009b, Yau10,ERetal,ESY2,Tao2011a}. A major focus of random matrix theory in recent years has been to broaden the reach to this vision to various models of ``not-too-random'' matrices whose entries have various \textit{dependencies}. Among the main well-established avenues of research are the study of probabilistically generated correlations \cite{AC,Erdos2015,AEKS,Banna13,Boutet96,Khorunzhy94} and graph-theoretically induced constraints on random adjacency matrices \cite{BHKY,Bau15,HLa}. The motivation to move towards increasingly structured ensembles is the widely held belief that the \textit{universality phenomenon} encompasses many other strongly correlated systems than traditional random matrices. This belief is based on many real-world examples of strongly correlated point processes which empirically reproduce random matrix statistics. Famous examples include the BGS quantum chaos conjecture and Montgomery's pair correlation conjecture for the zeros of the Riemann zeta function. It is major open problem to explain the breadth of the universality of random matrix statistics.

In our ongoing research program \cite{ALY,AL}, we turned to dynamical systems theory as a method for producing matrices with dependencies. We consider \textit{dynamically defined} matrices which are generated by evaluating a complex-valued function $f$ along orbits of an ergodic transformation $T$ and then using the sequence $f(x),f(Tx),f(T^2x),\ldots$ to fill a rectangular array, e.g., as in \eqref{eq:matrix} below. As is common in dynamical systems theory, all the randomness then comes from sampling the starting position $x$. 

We briefly recall that mathematical physicists have long studied the spectral theory of dynamically defined matrices of Schr\"odinger type. These matrices are typically tridiagonal, thus sparse.
In the Schr\"odinger (or more generally, Jacobi) world, many important works have tied together dynamical systems theory and spectral theory in deep and sometimes surprising ways, with a particular role played by the Almost Mathieu operator; see \cite{AJ,BGS01,BS,DKS,JL} for examples of breakthrough results. One finding has been that sufficiently random-like Schr\"odinger operators display Anderson localization and also Poisson spectral statistics \cite{Minami}. Hence, roughly speaking, it can be said that the sparse dynamically defined Schr\"odinger operators tend to belong the opposite of the random matrix regime.

We take a new perspective and consider dynamically defined Hermitian matrices that are \textit{full} (i.e., all entries are of comparable size). This puts us in the world of random matrix theory, but with a novel kind of correlation structure. Our goal is then to establish that these dynamically defined matrices display random matrix statistics down to the smallest scale, i.e., that dynamical correlations can mimic random ones.

The \textit{main result} of this work can be summarized as follows.\\ 

\textit{The GUE universality class contains dynamically defined matrices.}\\ 

We emphasize that a single entry $f(T^k x)$ of our  dynamically generated matrix (defined in \eqref{eq:matrix} below) fully determines all other entries, so in some sense there is complete deterministic dependence within the matrix and this fact places the model outside of the existing techniques. From a dynamical systems perspective, it is nonetheless clear that, say, a strongly mixing dynamical system will lead to rapid decorrelation of the entries. The latter perspective is fruitful for us, but we have to make precise how the dynamical decorrelation of matrix entries manifests on the spectral level and this requires some new techniques.

Naturally, the precise manner of dynamical decorrelation depends on the underlying dynamical system. In the prior works \cite{ALY,AL}, we have focused on extremely rigid dynamical systems of skew-shift type which are heavily correlated; the resulting spectra are difficult to analyze down to small spectral scales and universality remains open. Even in the Schr\"odinger world, the skew-shift still presents a significant frontier \cite{BS,HLS1,K1}.

The present paper instead starts from an especially natural strongly mixing dynamical system, the doubling map $T:[0,1] \to [0,1]$ defined by
\begin{equation}\label{eq:doublingdefn}
    T(x) = 2x \mod 1.
\end{equation}
and derive the desired strong conclusion---universal eigenvalue statistics. 

Existing techniques on correlated matrices come in two flavors: either they use a special model-dependent correlation structure (e.g., in the random graph setting) \cite{BHKY,Bau15,HLa} or they require assumptions on the correlation decay \cite{AEKS,Che2016,EKS} that are not fulfilled in the dynamical setting for the reason mentioned above that a single matrix entry fully determines all other ones. We overcome these limitations by using a resampling trick and by adapting techniques developed for correlated matrices with finite range of dependence \cite{Che2016} to a logarithmic range of dependence.

From the random matrix theory perspective, our result broadens the scope of the GUE universality class to encompass dynamically defined matrices. From a mathematical physics perspective, it provides a spectral-theoretic confirmation of the quasi-random nature of the doubling map; this can be seen as a ``delocalization analog'' of a well-known result by Bourgain-Schlag \cite{BS} establishing Anderson localization for potentials dynamically defined via strongly mixing potentials; see also \cite{CS}. We leave it as an open problem to extend the result presented here to a wider class of dynamical systems satisfying a quantitative mixing assumption, such as appropriate subshifts of finite type described by the potential formalism of Bowen \cite{Bowen}. Our result  opens the door to developing a much more general theory which connects dynamical systems theory to the spectra of dynamically defined full matrices. It would be interesting to see which features of such a theory mirror developments in the world of Schr\"odinger operator, especially given that they typically model the physically different localization regime.

The paper is organized as follows.
\begin{itemize}

\item In Section \ref{sec:model}, we define our model of dynamically defined matrices and list our main results on a local law for the Green's function, the universality of its eigenvalue statistics, and the delocalization of eigenvectors.
\item In Section \ref{sec:Step1}, we introduce our resampling method, list our analogous universality results for our resampled matrices, and show that the results of section \ref{sec:model} can be derived from the results of this section.
\item In Section \ref{sec:self-consist}, we review the importance of deriving a self-consistent equation for the Stieltjes' transform of our resampled matrices. We show that our model has a scalar self-consistent equation, and we derive this scalar self-consistent equation from an associated matrix self-consistent equation.
\item In Section \ref{sec:matrissc}, we derive our matrix self-consistent equation via appropriate concentration estimates.
\item In Section \ref{sec:analysis}, we prove the stability of our matrix self-consistent equation and derive associated error bounds for the Stieltjes' transform of our resampled matrix model, as well as the Green's function of our resampled matrix model.
\item In Section \ref{sec:locallaw}, we use our stability bounds and error estimates to prove a local law via an inductive scheme.
\item In Section \ref{sec:OU}. we compare our resampled matrix model to a model with a small Gaussian part via an Ornstein-Uhlenbeck process. This allows us to prove our main universality result.

\end{itemize}
\section{Model and main result} \label{sec:model}

\subsection{Dynamically defined matrices}
Let $\T=\R/\Z$ denote the standard torus. Given an \textit{evaluation (or sampling) function} $f:\T\to\C$ and a fixed starting point $x\in [0,1]$, we consider the \textit{dynamically defined matrix} 
$$
X_{ij}=\frac{1}{\sqrt{N}} f(T^{(2N-1)i+j} x),
$$
that is,
    \begin{equation} \label{eq:matrix} X=\frac{1}{\sqrt{N}}\begin{pmatrix}
  f( T x) & f( T^2 x) & \ldots & f(T^{N}x) \\
f(T^{2N+1} x) & f(T^{2N+2}x) & \ldots &f(T^{3N}x) \\
 f(T^{4N+1}x) & f(T^{4N+2}x) & \ldots & f(T^{5N}x)\\
\vdots & \vdots & \ddots & \vdots  \\
f(T^{2N^2-N+1}x) & f(T^{2N^2-N+2}x) & \ldots  & f(T^{2N^2}x)
\end{pmatrix} \end{equation}

Observe that all the entries of $X$ belong to a single dynamical orbit---the orbit of the starting point $x\in [0,1]$ under the doubling map $T$.  The normalization factor $\frac{1}{\sqrt{N}}$ is standard and ensures that the limiting spectral distribution is supported on an order-$1$ set.
For convenience, we start at the top left with $Tx$ instead of $x$ and we skip $N$ discrete time steps when moving from one row to the next. 

Our results concern the (real-valued) spectrum of the following Hermitization $H_X$ of $X$.
\begin{equation}\label{eq:Hdefn}
    H_X= \begin{pmatrix} 
    0 & X \\
     X^{\dagger} & 0\end{pmatrix}
\end{equation}
with $X^\dagger$ the adjoint of $X$. 

To turn this into a random matrix ensemble, we sample the starting position $x\in [0,1]$ from the uniform measure on $[0,1]$, which we call $\mathbb P$. We write $\mathbb E$ for the associated expectation. The uniform measure is a natural choice because it is the equilibrium measure of the doubling map.

We now state an informal version of the main result. See Theorem \ref{thm:main} below for the formal version.

\be{thm*}[Main result, informal version]
Suppose that $f$ is an admissible evaluation function in the sense of Definition \ref{defn:admissible}. Then, for every $k\geq 1$, the $k$-point correlation functions of the eigenvalues of $H_X$ converge to the $k$-point correlation functions of the eigenvalues of an $N\times N$ GUE matrix as $N\to\infty$.
\e{thm*}

The function $f(x)=\exp(2\pi \ti x)$ is an example of an admissible evaluation function in the sense of Definition \ref{defn:admissible}. 

 The result can be rephrased to describe the singular values of $X$ if desired. Generalization to non-square matrices $X$ is straightforward.

\subsection{Admissible evaluation functions}
The following assumption on the evaluation function $f$ arises naturally in the proof.

We require $f\in C^2(\mathbb T)$ and define its Fourier coefficients by 
\begin{equation}
    c_k =\int_0^1 f(y) e^{-2\pi \ti k y}\d y,\qquad k\in\Z.
\end{equation}
Then we associate to $f$ the function $g_f:\T\to[0,\infty)$ given by

\begin{equation}\label{eq:gfdefn}
    g_f(x)= \sum_{\substack{n\geq 0:\\ 2 \nmid n}} \left| \sum_{k=0}^{\infty} c_{n 2^k } \exp[2\pi \textnormal{i} k x]  \right|^2 
\end{equation}
where $2 \nmid n$ means that $2$ is not a divisor of $n$. To see that $g_f(x)$ is finite, note that $g_f(x)\leq \|c\|_{\ell^1}^2\leq C\|f\|_{H^1}^2$ with $c=(c_k)_{k\in\Z}$, where the last estimate uses the Cauchy-Schwarz inequality.

\begin{defn}[Admissible evaluation function] \label{defn:admissible}
We say that $f\in C^2(\T)$ is an \textit{admissible evaluation function} if $\mathbb E[f]=0$ and there exists a constant $g_{\min}>0$ such that 
\beq\label{eq:admissible}
\inf_{x\in[0,1]}g_f(x)\geq g_{\min} >0.
\eeq
\end{defn}


\be{ex}
The functions $f(x)=e^{2\pi \ti x}$ and $f(x)=\cos (2\pi x)$ are admissible evaluation functions.
\e{ex}


\begin{rmk}


A simple sufficient condition for $f\in C^2(\T)$ to be an admissible evaluation function can be obtained by ignoring cancellations and requiring some concentration of Fourier modes along dyadic scales. More precisely, if there exists an integer $n\geq1 $ such that $2 \nmid n$ and
\begin{equation}
    |c_n|  > \sum_{k=1}^{\infty} |c_{n2^k}|
\end{equation}
then $f$ is an admissible evaluation function.
\end{rmk}

To clarify the meaning of the function $g_f$, we express it through the following function $\phi:\Z\to\C$ which measures the correlation between $f$ and $\bar f$ among different dyadic scales:
\begin{align}
\label{eq:phidefn}
    \phi(j) = \mathbb{E}\l[f(x) \overline{f(T^j x)}\r]= \sum_{k=1}^{\infty} \oline{c_k} c_{k 2^j}.
\end{align}
We can use $\phi$ to define the infinite Toeplitz matrix $\Phi$ by
\beq\label{eq:Phidefn}
\Phi_{i,j}=\phi(i-j),\qquad i,j\in\Z.
\eeq
Observe that $\phi(j)=\overline{\phi(-j)}$ so $\Phi$ is Hermitian. We will see below that $\Phi$ determines the limiting spectra distribution of {{ $H_X$}} as $N\to\infty$

By expanding the square in \eqref{eq:gfdefn}, we can now rewrite $g_f$ as a Fourier series associated to the correlation function $\phi$, i.e.,
\beq\label{eq:gfrewrite}
g_f(x)
=\sum_{j\in\Z}^{\infty} \phi(j) e^{2\pi \ti j x}
= \phi(0) +2\mathrm{Re} \sum_{j=1}^{\infty} \phi(j) e^{2\pi \ti j x}.
\eeq
One says that formula \eqref{eq:gfrewrite} represents the \textit{Fourier symbol} of the infinite Toeplitz matrix $\Phi$. In particular, $\inf_x g_f(x)=\inf\mathrm{spec}\,\Phi$. We see that Assumption \eqref{eq:admissible} is equivalent to the spectral condition that $\Phi$ is \textit{strictly} positive definite. This is how the function $g_f$ arises in the proof.



\subsection{First result: local law}

We begin by stating a \textit{local law} which shows that, as $N\to\infty$, the empirical spectral distribution of $H$ converges to a deterministic limiting density even at small scales. The local law is a key prerequisite to our main result, universality.


The local law is conveniently formulated through the Stieltjes transform. Given a measure $\d\mu$ on the real axis, its Stieltjes transform $m_\mu$ is defined by
$$
m_\mu(z)=\int_{\R} \frac{1}{x-z}  \d \mu(x),\qquad \mathrm{Im} z>0.
$$ 

We prove a formulation of the local law in which the Stieljtes transform of $H_X$ converges to $m_\infty(z)$ which solves the following self-consistent equation
\beq\label{eq:scnaive}
\int_0^1 \frac{1}{g_f(x) m_\infty(z)+z}\d x=-m_\infty(z).
\eeq

\be{prop}\label{prop:minfty}
There exists a unique solution $m_\infty(z)$ to \eqref{eq:scnaive} with positive imaginary part. 
\e{prop}

This proposition is proved in Appendix \ref{app:limiting}. To state the local law, we introduce the following notation. Given $\eps,\de>0$, we let 
$$
\mathcal D_{\eps,\de}=\setof{E+i\eta \in \C}{\rho_\infty(E)\geq \eps,\quad \eta\in (N^{-1+\de},1)}
$$
where the first condition $\rho_\infty(E)\geq \eps$ puts us in the bulk of the spectrum. We say that a sequence of events $A_N$ holds \textit{almost surely as $N\to\infty$}, if $\mathbb P(A_N)\to 1$ as $N\to\infty$, where we recall that $\mathbb P$ is the uniform measure for $x\in [0,1]$.

\begin{thm}[Local law] \label{thm:LocalLaw}
Let $f\in C^2(\T)$ be an admissible evaluation function. Then, for every $\eps,\de>0$, 
\beq\label{eq:LL}
    \sup_{z\in\mathcal D_{\kappa,\de}} \l| \frac{1}{2N} \mathrm{Tr}{(H_X -z)^{-1}} - m_{\infty}(z)\r| \le \frac{N^{\delta}}{N \eta}
    \eeq
holds almost surely as $N\to\infty$.
\end{thm}

Note that the right-hand side in \eqref{eq:LL} vanishes in the limit as long as $\eta\gg N^{-1+\de}$ with $\de>0$ arbitrary, so the local law descends arbitrarily close to the optimal scale $N^{-1}$ where individual eigenvalues become resolved.

We can identify $m_\infty(z)$ as the Stieltjes transform of a limiting spectral distribution which we call $\rho_\infty(x)$. This $\rho_\infty(x)$ is a deterministic probability density function which generalizes the Wigner semicircle law $\rho_{\mathrm{sc}}(x)=\frac{1}{2\pi}\sqrt{4-x^2}\mathbbm 1_{[-2,2]}(x)$ to our correlated setting (see Example \ref{ex:sc} below). 

\be{prop}\label{prop:rhoinfty}
The function $m_\infty(z)$ is the Stieltjes transform of a continuous measure $\rho_\infty(x)\d x$, i.e.,
\beq
m_\infty(z)=\int_{\R}\frac{1}{x-z} \rho_\infty(x)\d x,\qquad \mathrm{Im}\, z>0.
\eeq
By the Stieltjes inversion formula, 
$$
\rho_\infty(E)=\lim_{\eta\to0} \frac{1}{\pi}\mathrm{Im}\, m_\infty(E+i\eta),\qquad E\in\R.
$$
\e{prop}

The proof of this proposition is deferred to Appendix \ref{app:limiting}. Notice that the density $\rho_\infty(x)$ is deterministic and only depends on the choice of evaluation function $f$.

\be{ex}\label{ex:sc}
For the evaluation function $f(x)=\exp(2\pi \ti x)$, we have $g_f(x)=1$ and $\rho_\infty(x)=\rho_{\mathrm{sc}}(x)=\frac{1}{2\pi}\sqrt{4-x^2}\mathbbm 1_{[-2,2]}(x)$, the Wigner semicircle law. See Figure \ref{fig:GlobalLaw} for a pictorial representation of the emergent global law in this special case of Theorem \ref{thm:LocalLaw}.
\e{ex}

\begin{figure}
    \includegraphics[scale=0.75]{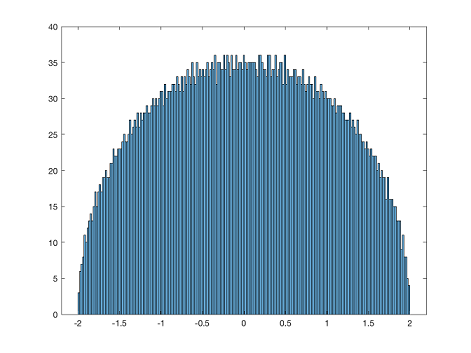}
    \caption{The empirical spectral distribution of matrices given by \eqref{eq:matrix} with $f(x)=\exp(2\pi \ti x)$. This histogram depicts the emergence of the Wigner semicircle law on the global scale. Theorem \ref{thm:LocalLaw} also establishes that this convergence continues to hold down to small scales.}
        \label{fig:GlobalLaw}
\end{figure}

\begin{rmk}
Let us explain the form of the self-consistent equation \eqref{eq:scnaive}. In the proof, it arises in the form
\beq\label{eq:sc}
-\frac{1}{m_\infty(z)}m_{\Phi}\l(-\frac{z}{m_\infty(z)}\r)=m_\infty(z).
\eeq
where $m_{\Phi}\equiv m_{\rho_{\Phi}}$ is the Stieltjes transform of the spectral density of the infinite Toeplitz matrix $\Phi$. Since $\Phi$ is unitarily equivalent to multiplication by $g_f$ in Fourier space, we have the explicit representation
$$
m_{\Phi}(z)=\int_{0}^1 \frac{1}{g_f(x)-z}\d x,\qquad \mathrm{Im}\, z>0,
$$
which connects \eqref{eq:sc} to \eqref{eq:scnaive}.
\e{rmk}


\subsection{Main result --- Universality}
We are now ready to state our main universality result. We write $\lambda_1,\ldots, \lam_{2N}$ for the eigenvalues of { $H_X$} and $p_N^N:\R^{2N}\to\R$ for the symmetrized eigenvalue density of $H_X$. The $k$-point correlation functions $p_N^k$ are defined by
$$
p_N^k(\lam_1,\ldots,\lam_k)=\int p_N^N(\lam_1\,\ldots,\lam_N)\d\lam_{N-k+1}\ldots\d\lam_{N},\qquad k\geq 1.
$$

 The universal objects are $p^{(k)}_{N,\mathrm{GUE}}$, defined analogously as the $k$-point correlation functions of a matrix from the Gaussian Unitary Ensemble. They are explicitly computable, e.g.,
 $$
p^{(k)}_{N,\mathrm{GUE}}(x_1,\ldots,x_k)=\det(K_N(x_i,x_j))_{1\leq i,j\leq k}
$$
where $K_N(x,y)$ has an explicit formula in terms of Hermite polynomials \cite{bYau}.

We use the convention that for $\bm{\al}=(\al_1,\ldots,\al_k)\in \R^k$,
$$
E+\bm{\al}=(E+\al_1,\ldots,E+\al_k).
$$

\begin{thm}[Main result --- Universality] \label{thm:main}
Let $f\in C^4(\mathbb T)$ be an admissible evaluation function. Let $k\geq 1$ and let $O:\mathbb{R}^k \to \mathbb{R}$ be smooth and compactly supported. Then, for $\eps>0$ and $\rho(E)\geq \eps$, there exists $\kappa>0$ such that
\begin{equation}\label{eq:main}
        \int_{\mathbb{R}^k} O(\bm{\al})
        \left[p_N^k\left(E+\frac{\bm{\al}}{2N \rho_\infty(E)}\right)-p_{N,\mathrm{GUE}}^k\left(E+\frac{\bm{\al}}{2N \rho_{\mathrm{sc}}(E)}\right) \right]
       \d^k\bm{\al}        
= \textnormal{O}(N^{-\kappa})
\end{equation}
\end{thm}

\begin{rmk}
\begin{enumerate}[label=(\roman*)]
\item The condition $\rho_f(x)\geq \eps$ means that we work in the bulk of the spectrum. It may be possible to prove similar results at the spectral edge, but this is not our focus here.
\item 
While the results are stated as holding almost surely as $N\to\infty$, the proof quantifies the convergence speed as polynomial. Since this point is not our focus here, we use the qualitative notion for simplicity of presentation.
\item There exist universality result for correlated matrices in the literature \cite{Che2016,EKS}. However, these do not apply to our dynamically defined model. For instance, the condition on correlation decay in \cite{EKS} is that for any pair of bounded functions $F$ and $G$ applied to entries $H_{ij}$ and $H_{kl}$ ,
 \begin{equation}\label{eq:eks}
    \mathbb{E}[F(H_{ij}) G(H_{kl})] \le C \|F\|_{\infty} \|G\|_{\infty} \frac{1}{(|i-k| + |j -l|)^{12}}.
\end{equation}
For our model, fixing $i,j,k,l$ and taking $F=\mathrm{id}$ and $G=T^{p}$ for an appropriate power $p$, the left-hand side equals $\mathbb E[|H_{kl}|^2]=N^{-1}$, a constant. Since this does not decay at all, the correlation condition \eqref{eq:eks} breaks down in our model.
\end{enumerate}
\end{rmk}

\be{rmk}[Open problem]
The formulae \eqref{eq:matrix} and \eqref{eq:Hdefn} defining our matrix ensemble can be generalized to dynamically define a matrix ensemble starting from \textit{any} dynamical system $(X,T,\mu)$. It is an interesting open problem, lying on the interface of dynamical systems theory and random matrix theory, to extend the present universality result to a much broader class of dynamical systems. Natural candidates are dynamical systems satisfying a quantitative mixing assumption, e.g., subshifts of finite type building on the potential formalism developed in \cite{Bowen}; see also \cite{ADZ}.
\e{rmk}

\subsection{Delocalization of eigenvectors}
In physics, the delocalization of eigenvectors is seen as a hallmark of the random matrix theory phase. In a standard way, the precise control on Green's function established in proving of Theorem \ref{thm:LocalLaw} implies eigenvector delocalization bounds in the bulk of the spectrum.

We use the convention that eigenvectors $u_\al\in \R^N$ are $\ell^2$-normalized and so a completely delocalized eigenvector, which is roughly equally supported on all coordinates, would have $\max_{1\leq i\leq N} |u_\alpha(i)|^2\sim N^{-1}$.

\begin{thm}[Delocalization bound for eigenvectors]\label{thm:deloc}
Let{  $H_X u_\alpha=Eu_\alpha$ }for some $E\in\R$ satisfying $\rho_\infty(E)\geq \eps$ for some $\epsilon>0$. Then, there exists $\de>0$ such that
$$
\max_{1\leq i\leq N} |u_\alpha(i)|^2 \leq N^{\de-1} 
$$
holds almost surely as $N\to\infty$.
\end{thm}

\subsection{Proof strategy}

Overall, our proof relies on the famous $3$-step strategy for proving universality developed in the past years. However, the dynamical nature of correlations in our matrix ensemble requires some new ideas which can be summarized as follows.
\be{itemize}
\item The first novel feature of our proof is a \textit{dynamical resampling trick}
 which we use to ``preprocess'' the dynamically defined ensemble $H_X$ into another one, called $H_Y$, whose entries become independent after $(\log N)^6$ many discrete time steps. The resampling procedure is based on the binary expansion because the doubling map then acts as a digit shift. Importantly, the construction is such that the Green's functions before and after resampling agree on scales $\textnormal{Im}[z] \ll N^{-1}$. This means that their spectra are essentially the same and so it suffices to prove local law and universality for the ensemble $H_Y$ with its logarithmic range of dependence.
 
 \item In a second step, we prove local law and universality for the eigenvalues of  {$H_Y$} with its logarithmic range of dependence by extending Che's analysis for finite range of dependence \cite{Che2016}. The possibility that these methods could be extended to range of dependence to be logarithmic, even to a logarithmic power of a logarithm, was already raised by Che.
 \item A problem compared to the situation studied in \cite{Che2016} is that we no longer have strict positivity of the correlation matrix and Lipschitz continuity of the correlations. We address these issues by leveraging the block structure and symmetries to reduce to a scalar self-consistent equation for the Stieltjes transform. For this scalar equation, we can recover positive definiteness by using the Ward identity, the assumption that $f$ is an admissible evaluation functions and techniques in the spectral theory of Toeplitz matrices, e.g., turning a banded Toeplitz matrix into a circulant matrix (whose spectrum is explicitly computable) by a finite-rank perturbation.
 \item Other building blocks of the proofs, e.g., the fast equilibration of Dyson Brownian Motion, are standard in the field and are accordingly only summarized briefly.  
 
  \e{itemize}

We use the convention that $C$ denotes a generic positive constant that is independent of $N$ and whose value may change from line to line.

\section{Step I: Resampling to logarithmic range of dependence} \label{sec:Step1}
\subsection{Resampling binary digits}
Recall that $x$ is the starting point of the doubling map dynamics in \eqref{eq:matrix}. We write
$$
x= \sum_{n\geq 1}d _n 2^{-n},\qquad d_n\in\{0,1\},
$$ 
for its binary expansion. The doubling map acts as a digit shift in the binary basis,
$$
T^k x=\sum_{n\geq 1} d_{n+k} 2^{-n},\qquad k\geq 0
$$
We let $\{b^k_n\}_{n,k\geq 1}$ be a two-parameter family of independent coin flips with outcomes $\{0,1\}$ occurring with probability $\tfrac{1}{2}$ each. We then define the resampling of $T^k x$ by
\beq\label{eq:ykdefn}
y_k= \sum_{1\leq n\leq (\log N)^6} d_{n+k} 2^{-n} +\sum_{n\geq (\log N)^6+1} b_n^k 2^{-n},\qquad \textnormal{where } \log N\equiv \log_2 N.
\eeq
(The choice of $6$ is somewhat arbitrary; any sufficiently large power of a logarithm will do.)

The resampled matrix ensemble is given by
\begin{equation} \label{eq:Ydefn}
Y=     \frac{1}{\sqrt{N}}
\begin{pmatrix}
 f(y_1) & f(y_2) &\ldots &f(y_N)\\
 f(y_{2N+1}) & f(y_{2N+2}) &\ldots &f(y_{3N})\\  
\vdots & \vdots & \ddots & \vdots  \\
 f(y_{2N^2-N+1}) & f(y_{2N^2-N+2}) &\ldots &f(y_{2N^2})\\  
\end{pmatrix}\end{equation}
with Hermitization
\begin{equation} \label{eq:Tildmat}
H_Y = \begin{pmatrix} 
    0 & Y\\
    Y^{\dagger} & 0 
    \end{pmatrix}.
\end{equation}

\subsection{Local law and universality for the resampled ensemble}
The advantage of the resampled ensemble $H_Y$ is that in contrast to the original ensemble $H_X$ it  only has logarithmic range of dependence: for $|a-b|>(\log N)^6$, the random variables $y_a$ and $y_b$ are independent. In Step 2, we will use this fact to prove that the analogs of our main results (local law and universality) indeed hold for $H_Y$. For completeness, they are stated here.

\begin{thm}[Local law after resampling] \label{thm:auxLocalLaw}
Let $f\in C^2(\T)$ be an admissible evaluation function. Then, for every $\eps,\de>0$, 
\beq\label{eq:LL0}
    \sup_{z\in\mathcal D_{\kappa,\de}} \l| \frac{1}{2N} \mathrm{Tr}{(H_Y -z)^{-1}} - m_{\infty}(z)\r| \le \frac{N^{\delta}}{N \eta}
    \eeq
holds almost surely as $N\to\infty$.
\end{thm}

We write $\tilde p_N^k$ for the $k$-point correlation function associated to the eigenvalues of $H_Y$.

 \begin{thm}[Universality after resampling] \label{thm:auxmain}
Let $f\in C^4(\mathbb T)$ be an admissible evaluation function. Let $k\geq 1$ and let $O:\mathbb{R}^k \to \mathbb{R}$ be smooth and compactly supported.
Then, for $\eps>0$ and $\rho(E)\geq \eps$, there exists $\kappa>0$ such that
\begin{equation}\label{eq:mainresampling}
        \int_{\mathbb{R}^k} O(\bm{\al})
        \left[\tilde p_N^k\left(E+\frac{\bm{\al}}{2N \rho_\infty(E)}\right)-p_{N,\mathrm{GUE}}^k\left(E+\frac{\bm{\al}}{2N \rho_{\mathrm{sc}}(E)}\right) \right]
       \d^k\bm{\al}        
=\mathcal{O}(N^{-\kappa})
\end{equation}
as $N\to\infty$.
\end{thm}

Before we prove these theorems, we use them to derive Theorems \ref{thm:LocalLaw} and \ref{thm:main}.

\subsection{Green's function comparison lemma}
Here and in the following we denote
$$
G_X(z)= (H_X-z)^{-1},\quad G_Y(z) = (H_Y-z)^{-1}, \qquad \mathrm{Im}\, z>0
$$
and
$$
m_X(z)=\frac{1}{2N}\mathrm{Tr} G_X(z),\qquad m_Y(z)=\frac{1}{2N}\mathrm{Tr}G_Y(z).
$$
From now on, we always assume that $f\in C^2(\T)$ is an admissible evaluation function.

\begin{lm}[Green's Function Comparison Lemma]\label{lm:GFCL}
Let $n\geq 1$, $\eps>0$ and let $E_1,\ldots,E_n$ satisfy { $\rho_\infty(E_i)\geq \eps$ and $\rho_{\infty}$ is the measure associated to $m_{\infty}$}. Given $\sigma_1,\ldots,\sigma_n\leq \sigma$, we set
$$
z_j=E_j+\ti \eta_j,\qquad \eta_j=N^{-1-\sigma_j}.
$$
Then, there exists $C_\sigma>0$ depending only on $\sigma$ such that
$$
\l|\prod_{k=1}^n\mathrm{Im}\, m_X(z_k)-\prod_{k=1}^n\mathrm{Im}\, m_Y(z_k)\r|\leq \frac{C_\sigma}{N}.
$$
   \end{lm}
   
   \begin{proof}[Proof of Lemma \ref{lm:GFCL}]
Let $\sigma>0$. By the resolvent identity, we have that
\begin{equation}
    G_X(z) - G_Y(z) = G_Y(z) (X - Y) G_X(z).
\end{equation}
The entries of the matrix $(X - Y)$ are of the form $f(T^k x) -y_k$. For these, we have the bound,
\begin{equation}
    |f(T^k x) - f(y_k)| \le \| f' \|_{\infty} |T^k x -y_k| \le \| f' \|_{\infty} 2^{-(\log N)^6}.
\end{equation}
The outside factors are controlled by $|G_{X,ij}(z_k)|,|G_{Y,ij}(z_k)| \le \frac{1}{\textnormal{Im}[z_k]} \le N^{1+\sigma} $. After taking the trace, this shows
$$
|m_X(z_k)-m_Y(z_k)|\leq \| f' \|_{\infty} N^{4+2\sigma } 2^{-(\log N)^6}.
$$
Using that $|m_X(z_k)|,|m_Y(z_k)|\leq N^{1+\sigma}$, we conclude
$$
\l|\prod_{k=1}^n\mathrm{Im}\, m_X(z_k)-\prod_{k=1}^n\mathrm{Im}\, m_Y(z_k)\r|
\leq \| f' \|_{\infty} N^{(n-1)\sigma}  N^{4+2\sigma } 2^{-(\log N)^6}\leq \frac{C_\sigma}{N}
$$
since the logarithm comes with a power. 
\end{proof}

\subsection{Proof of the main result assuming Theorems \ref{thm:auxLocalLaw} and \ref{thm:auxmain}}
\be{proof}[Proof of the local law, Theorem \ref{thm:LocalLaw}]
This follows by combining Lemma \ref{lm:GFCL} and Theorem \ref{thm:auxLocalLaw}.
   \e{proof}
   
For Theorem \ref{thm:main}, we need another technical comparison result.

   \be{lm}[Correlation Function Comparison Lemma]
   \label{lm:CFCL}
Let $k\geq 1$ and let $O:\mathbb{R}^k \to \mathbb{R}$ be smooth and compactly supported.
Then, for $\eps>0$ and $\rho(E)\geq \eps$, there exists $\kappa>0$ such that
\begin{equation}
    \int_{\mathbb{R}^k} O(\bm{\alpha}) \l(p_{N}^{(k)}\l(E+\frac{\bm{\al}}{2N}\r) - \tilde p_N^{(k)}\l(E+\frac{\bm{\al}}{2N}\r) \r)\d^k\bm{\al}=\mathcal{O}(N^{-\kappa}),
\end{equation}
as $N\to\infty$.
\e{lm}
   
   \be{proof}
   The derivation of Lemma \ref{lm:CFCL} from Lemma \ref{lm:GFCL} and the local laws is standard for generalized Wigner matrices; see \cite[Thm 15.3]{bYau} and \cite[Thm 6.4]{Erdos2012}. In essence, the proof idea is to mollify the observable $O$ in each argument through approximate $\de$-functions of the form $\mathrm{Im}\,\frac{1}{x-E-i\eta}$ with $\eta\sim N^{-1-\sigma}$, which is a good approximation because these small scales resolve individual eigenvalues. This allows to express the tested correlation function using polynomials in $\mathrm{Im}\, m(z_1),\ldots, \mathrm{Im}\, m(z_k)$ to which one applies Lemma \ref{lm:GFCL}. An inclusion-exclusion argument takes care of the possibility of eigenvalues matching. The local laws for { $m_Y(z)$ and $m_X(z)$} play the role of a priori bound on the eigenvalue density in the form of \cite[eq.\ (15.7)]{bYau}.  
   
   Our contribution here is solely to observe that the arguments carry over verbatim to the ensemble defined by{ $H_Y$}. For the details, we refer the interested reader to \cite{bYau}[Thm 15.3] and \cite[Thm 6.4]{Erdos2012}.
   \e{proof}
   
      \be{proof}[Proof of the main result, Theorem \ref{thm:main}]
      This follows by combining Lemma \ref{lm:CFCL} and Theorem \ref{thm:auxmain}.
   \e{proof}

%
%

It remains to prove Theorems \ref{thm:auxLocalLaw} and \ref{thm:auxmain}. The main work is to establish Theorem \ref{thm:auxLocalLaw}, the rest is handled by the well-developed Dyson Brownian Motion machinery. From now on, we only consider the resampled matrices $Y$ and we denote $H\equiv H_Y$, $m\equiv m_Y$, { and $G\equiv G_Y$} .

\section{Derivation of the scalar self-consistent equation}\label{sec:self-consist}
This section is devoted to studying the self-consistent equation satisfied by the Stieltjes transform 
\beq\label{eq:mdefn}
m(z)=\frac{1}{2N}\mathrm{Tr}[(H-z)^{-1}].
\eeq
 We will see that it is essentially given by \eqref{eq:sc} up to error terms that are subleading in $N$.

The standard approach for matrices with independent entries is to remove a single row and column via the Schur complement formula. With dependencies, this is no longer sufficient. The solution found in \cite{Che2016} is to expand the rows and columns even further until sufficient decoupling is achieved. In this section, we adapt this approach to the dynamically defined ensemble.

\subsection{Preliminaries}
We introduce some notation.
Following \cite{Che2016}, we define the $2N^2\times 2N^2$ matrix $\Sigma$, called the \textit{full correlation matrix of $H$}, 
\beq\label{eq:Sigmadefn}
\Sigma_{(i,j),(k,l)}=\xi_{ijkl}=N\mathrm{Cov}\l[H_{ij},\overline{H_{kl}}\r], \qquad 1\leq i,j,k,l \leq 2N.
\eeq
One can immediately see that the matrix $\Sigma$ is not strictly positive semidefinite in the sense of \cite[Def. 2.2]{Che2016} since $\Sigma_{(i,j) (i,j)}=0$ for $1\le i,j\le N$. 

To quantify the error terms in the self-consistent equation, we need some notation.

\begin{defn}[Notation for index reduction]
Let $A$ be a matrix indexed by a subset of the integers.
\begin{enumerate}[label=(\roman*)]
    \item For $T$ and $S$ subsets of $\Z$, we write $A\vert_{T,S}$ for the submatrix of $A$ whose rows are given by $T$ and whose columns are given by $S$.
    \item For $U$ is a subset of $\Z$, we denote by $A^{(U)}$ the matrix
    $$
  A^{(U)}_{ij}= \left(A_{ij} \mathbbm{1}[i \in U] \mathbbm{1}[j \in U] \right).
$$
     This means that we set each entry of $A$ whose row or column index belongs to $U$ to 0 and do not alter the matrix $A$ otherwise. 
     \item For $U\subset \{1,\ldots,2N\}$, we also set $G^{(U)}(z) = (Y^{(U)} -z)^{-1}$. When $U$ is a single element $\{k \}$, then we write $G^{k}$ instead of { $G^{(\{k\})}$}.
\end{enumerate}

\end{defn}

\begin{defn}[Stochastic control parameters]\label{def:StCo}
In order to establish various error bounds, we would require the following two stochastic control parameters.
\begin{enumerate}[label=(\roman*)]
    \item We define \begin{equation*}
        \Gamma := 1 \vee\sup_{i,j \in \{1,\ldots,2N\}} |G_{ij}|.
    \end{equation*}
    \item Given $i\in\{1,\ldots,2N\}$, let $I$ and $J$ be disjoint subsets of 
    $$
    [ i- 2(\log N)^6, i + 2(\log N)^6]\cap\Z.
    $$
Let {$\| \cdot \|$} denote the operator norm on $\mathbb{C}^{|I|}$ with respect to the Euclidean norm. We define
\begin{equation*}
        \gamma:= 1 \vee \sup_{1\leq i\leq 2N} \sup_{I,J} \| {( G_{I,I}^{(J)})^{-1} } \|.
    \end{equation*}
\end{enumerate}
\end{defn}

\subsection{The self-consistent equation}
We first state the general form of the self-consistent equation satisfied by the Green's function.

\be{lm}[Self-consistent equation for the Green's function]\label{lm:scGF}
Let $z=E+i\eta$ and $\sigma>0$. Then
\begin{equation}\label{eq:lmscGF}
    -\frac{1}{N} \sum_{1\leq k,l,m\leq 2N}G_{ik} \xi_{kl jm} G_{l m} - z G_{ij} = \delta_{ij} + \mathcal{O}_{\sigma}\l(\frac{N^{2\sigma} \Gamma^5 \gamma^3}{\sqrt{N \eta}}\r).
\end{equation}

\e{lm}

This lemma is proved in the next section. The proof is similar to that of \cite[Lemma 3.10]{Che2016} due to the fact that the relevant concentration estimates for quadratic forms can be extended to a logarithmic range of dependence.

We are able to leverage the special structural properties of the dynamically defined matrix model to reduce to a scalar self-consistent equation for the Stieltjes transform $m(z)$.

The relevant finite-$N$ analog of ${\Phi}$ from \eqref{eq:Phidefn} is the banded $N\times N$ Toeplitz matrix $\Phi^N$ defined as follows, for any $1\leq i,j\leq N$,
\beq\label{eq:PhiNdefn}
\Phi^N_{i,j}=
\begin{cases}
\phi(i-j),\qquad &\textnormal{if } |i-j|\leq (\log N)^6,\\
0,\qquad &\textnormal{if } |i-j|\geq (\log N)^6+1.
\end{cases}
\eeq

We note that $\Phi^N$ is Hermitian because the function $\phi$ satisfies $\phi(j)=\oline{\phi(-j)}$.

{{  With the identification of the most-important parts of the covariance, we can simplify the self-consistent equation from \ref{lm:scGF} to have a nice block structure.
\be{cor}
Let $G^1= G|_{[1,N] \times[1,N]}$, $G^2=G|_{[N+1,2N] \times[N+1,2N]}$, $G^3= G|_{[1,N] \times[N+1,2N]}$, and $G^4= G|_{[N+1,2N] \times[1,N]}$.  We have the following equations on the blocks of $G^i$
\begin{equation} \label{eq:equationset}
\begin{aligned}
	&G^1 \left( -\frac{1}{N} \text{Tr}[G^2 \Phi^N]] -z\right) = \mathbbm{1} + E_1= \mathbbm{1} + \mathcal{O}_{\sigma}\l(\frac{N^{2\sigma} \Gamma^5 \gamma^3}{\sqrt{N \eta}}\r)\\
	& G^2\left(\frac{1}{N}\text{Tr}[G^1] \Phi^N - z\right) = \mathbbm{1} +E_2= \mathbbm{1} + \mathcal{O}_{\sigma}\l(\frac{N^{2\sigma} \Gamma^5 \gamma^3}{\sqrt{N \eta}}\r)\\
	& G^3\left(- \Phi^N \frac{1}{N}\text{Tr}[G^1] -z \right) =  E_3=\mathcal{O}_{\sigma}\l(\frac{N^{2\sigma} \Gamma^5 \gamma^3}{\sqrt{N \eta}}\r)\\
	& G^4\left( -\frac{1}{N} \text{Tr}[G^2 \Phi^N]] -z\right) = E_4= \mathcal{O}_{\sigma}\l(\frac{N^{2\sigma} \Gamma^5 \gamma^3}{\sqrt{N \eta}}\r),
\end{aligned}
\end{equation}
where the $\mathcal{O}_{\sigma}$ errors are elementwise.

By setting these errors to zero, we can derive the limiting matrix self-consistent equation.

If we let $\Xi(\mathcal{M})$ be the operator on $2N$ by $2N$ matrices that sends 
\begin{equation} \label{eq:defXi}
	\Xi(\mathcal{M})= \frac{1}{N} \sum_{\substack{(k,l,j,m) = (a,b+N,c,d+N) \\ \text{ or} (a+N,b,c+N,d) \\ 1 \le a,b,c,d \le N}} \mathcal{M}_{ik} \xi_{kljm} \mathcal{M}_{lm},
\end{equation}
Intuitively, we see that $G$ should approach the solution of the limiting self-consistent equation,
\begin{equation} \label{eq:firsteq}
 \mathcal{M} = (-z - \Xi(\mathcal{M}))^{-1}.
\end{equation}
\e{cor}

}
}

{ Unfortunately, this matrix self-consistent equation is still not in a good enough form for detailed analysis.  However, we notice vast simplications from the equations in \ref{eq:equationset}. The quations for $G^1$ and $G^2$ are closed and this is all that is necessary to determine the behavior of the trace of the Green's function. Furthermore, we would also expect $G^3$ ,$G^4$ and the off digaonal entries of $G^1$ to concentrate around zero. With this simplications, we can present a scalar self-consistent equation. This form of the self-consistent equation is the most important for deriving stability and error propagation bounds in section \ref{sec:analysis}.} 

We write $m_{\Phi^N}$ for the Stieltjes transform of the (no longer explicit) spectral measure of $\Phi^N$. Recall { the definition of $m(z)$ from Equation }\eqref{eq:mdefn}.  
\be{prop}[Scalar self-consistent equation for $m(z)$]\label{prop:sc}
We have
\beq\label{eq:propsc}
   -\frac{1}{m(z)} m_{\Phi^N}\l(-\frac{z}{m(z)}\r) =  m(z) + \mathcal{O}_{\sigma}\l(\frac{N^{2\sigma} \Gamma^6 \gamma^3}{\sqrt{N \eta}}\r).
   \eeq
\e{prop}

\be{rmk}
Formally taking $N\to\infty$, we see that we can expect { the above equation to approximate the limiting self-consistent equation}  \eqref{eq:sc} .
\e{rmk}

\subsection{Proof of Proposition \ref{prop:sc}}
We recall the Ward identity (or rather Ward identities since they helpfully decompose into two relations for sample covariance matrices such as $H$). These give us control on the sizes of the off-diagonal Green's function elements in terms of its diagonal elements. 

\begin{lm}[Ward Identities] \label{lem:Ward}
The Green's function satisfies
\begin{equation}
\begin{aligned}
    &\sum_{l=1}^N |G_{i,l}|^2 = \frac{1}{2} \frac{\textnormal{Im}[G_{ii}]}{\eta}\\
    & \sum_{l=1}^N |G_{i, l+N}|^2 = \frac{1}{2} \frac{\textnormal{Im}[G_{ii}]}{\eta}.
\end{aligned}
\end{equation}
\e{lm}

Here and in the following use both $A_{i,j}$ or $A_{ij}$ for entries of a matrix $A$, depending on convenience.

\be{proof}[Proof of Proposition \ref{prop:sc}]
\textit{Step 1.} Correlation matrix. Let $1\leq i,j,k,l\leq N$. Using \eqref{eq:ykdefn} and conservation of measure of the doubling map, we find for the correlation matrix
\beq\label{eq:Sigmaentries}
\begin{aligned}
\Sigma_{(i,j+N),(k,l+N)}
=&\mathbb E\l[f(y_{(2N-1)i+j}) \overline{f(y_{(2N-1)k+l})}\r]
=\de_{i,k}\Phi^N_{l,j},\\
\Sigma_{(i+N,j),(k+N,l)}
=&\mathbb E\l[f(y_{(2N-1)l+k)} \overline{f(y_{(2N-1)j+i})}\r]
=\de_{j,l}\Phi^N_{i,k},\\
\Sigma_{(i,j+N),(k+N,l)}
=&\mathbb E\l[f(y_{(2N-1)i+j)} f(y_{(2N-1)l+k})\r]
=\de_{i,l}\Psi^N_{k,j},\\
\Sigma_{(i+N,j),(k,l+N)}
=&\mathbb E\l[{ \overline{f(y_{(2N-1)j+i})} \overline{ f(y_{(2N-1)k +l})}}\r]
=\de_{j,k}{\overline{\Psi^N_{i,l}}},
\end{aligned}
\eeq
where $\Phi^N$ is given by \eqref{eq:PhiNdefn} and we introduced its sibling 
\beq\label{eq:PsiNdefn}
\Psi^N_{i,j}=
\begin{cases}
\psi(i-j),\qquad &\textnormal{if } |i-j|\leq (\log N)^6,\\
0,\qquad &\textnormal{if } |i-j|\geq (\log N)^6+1.
\end{cases}
\eeq
with
$$
\psi(j)=\mathbb{E}\l[f(x)f(T^j x)\r]= \sum_{k=1}^{\infty} c_k c_{k 2^j},\qquad j\in\Z.
$$
We note that \eqref{eq:Sigmaentries} contains all the only non-trivial entries of $\Sigma$, i.e., all others are zero. The error terms $\mathcal{O}(2^{-(\log N)^6})$ come from \eqref{eq:ykdefn} and the mean-value theorem; they are super-polynomially decaying and can thus be ignored in the following.

\textit{Step 2.} Reduction to blocks. We introduce the block Green's functions
$$
G^1 = G|_{\{1,\ldots,N\}\times \{1,\ldots,N \}} ,\qquad 
G^2= G|_{\{N+1,\ldots,2N\} \times \{N+1,\ldots, 2N\}}.
$$
and their normalized traces
$$
T_\al(z)=\frac{1}{N}\mathrm{Tr}[G^\al(z)],\qquad \al=1,2.
$$
The Schur complement formula and cyclicity readily imply that $T_1(z)=T_2(z)$ and consequently 
\beq\label{eq:mT1T2decompose}
m(z)=\frac{1}{2N}\mathrm{Tr}[G(z)]=\frac{T_1(z)+T_2(z)}{2}=T_1(z).
\eeq
We see that it suffices to study the reduced Stieltjes transform $T_1(z)$.

We decompose the self-consistent equation \eqref{eq:lmscGF} into blocks accordingly, taking into account when the correlation matrix vanishes. We first let $i,j\leq N$. Rearranging terms, we obtain
\begin{equation}\label{eq:scGF1}
\begin{aligned}
    &-\frac{1}{N} \sum_{k,l,m=1 }^NG^1_{ik} \xi_{k, l+N, j,m+N} G^2_{lm} - z G^1_{ij}\\ 
    =& \delta_{ij}+\frac{1}{N} \sum_{k,l,m=1 }^NG_{i,k+N} \xi_{k+N, l,j, m+N} G_{l,m+N} + \mathcal{O}_{\sigma}\l(\frac{N^{2\sigma} \Gamma^5 \gamma^3}{\sqrt{N \eta}}\r)\\
=& \delta_{ij}+ \mathcal{O}_{\sigma}\l(\frac{N^{2\sigma} \Gamma^5 \gamma^3}{\sqrt{N \eta}}\r)\\
\end{aligned}
\end{equation}
For the second step, we used the formula from Step 1, Cauchy-Schwarz, the Ward identity, and { $\|\psi\|_{\ell^2}<\infty$ }(proved below) to estimate
$$
\begin{aligned}
\frac{1}{N}\l| \sum_{k,l,m=1 }^NG_{i,k+N} \xi_{k+N, l, j,m+N} G_{l,m+N} \r|
\leq &\frac{1}{N} \sum_{k,m=1 }^N  |{ \overline{\psi(k-m)}} G_{i,k+N}G_{j,m+N}| \\
\leq& \frac{\Gamma}{2 { N}\eta} \|\psi\|_{\ell^2}
=\mathcal{O}_{\sigma}\l(\frac{N^{2\sigma} \Gamma^5 \gamma^3}{\sqrt{N \eta}}\r),
\end{aligned}
$$
To see that { $\|\psi\|_{\ell^2}<\infty$}, note that $\psi$ is the Fourier transform of
$$
\tilde g_f(x)= \sum_{\substack{n\geq 0:\\ 2 \nmid n}} \left( \sum_{k=0}^{\infty} c_{n 2^k } \exp[2\pi \textnormal{i} k x]  \right)^2, 
$$
which is dominated by $g_f$. Hence, we have ${\|\psi\|_{\ell^2}}\leq {\|\tilde{g}_f\|_{L^2([0,1])}}\leq \|g_f\|_{L^\infty([0,1])}\leq C\|f\|_{H^1(\R)}$.

Using the formula from Step 1 to rewrite the left-hand side of \eqref{eq:scGF1}, we conclude that
\beq\label{eq:<block}
 -\frac{1}{N} G^1_{ij}\sum_{l,m=1 }^N \Phi^N_{m,l} G^2_{lm} - z G^1_{ij}
 =\delta_{ij}+ \mathcal{O}_{\sigma}\l(\frac{N^{2\sigma} \Gamma^5 \gamma^3}{\sqrt{N \eta}}\r)
\eeq
For $i,j>N$ we have to consider the expression
$$
\begin{aligned}
     -\frac{1}{N} \sum_{k,l,m=1}^N G^2_{ik} \xi_{k+N,l,j+N,m} G^1_{lm} - z G^2_{ij} =& \delta_{ij}+ \mathcal{O}_{\sigma}\l(\frac{N^{2\sigma} \Gamma^5 \gamma^3}{\sqrt{N \eta}}\r).
\end{aligned}
$$

Completely analogous reasoning yields the remarkably different formula
\beq\label{eq:>block}
\begin{aligned}
     -T_1(z) \sum_{k=1}^N G^2_{ik}\Phi^N_{k,j} - z G^2_{ij} =& \delta_{ij}+ \mathcal{O}_{\sigma}\l(\frac{N^{2\sigma} \Gamma^5 \gamma^3}{\sqrt{N \eta}}\r).
\end{aligned}
\eeq
where we used that $\tfrac{1}{N}\sum_l G_{ll}^2=T_2(z)=T_1(z)$.

Inspecting the system of equations \eqref{eq:<block} and \eqref{eq:>block}, we notice that we can reduce $G^1$ to $T_1(z)$ also from \eqref{eq:<block} by averaging over $i=j$. By contrast, the elements of $G^2$ are multiplied by the Hermitian Toeplitz matrix $\Phi^N$. We obtain the system
\begin{equation}
\begin{aligned}
     T_1(z) \l(-\frac{1}{N}\mathrm{Tr}[G^2\Phi^N ] -z\r) =& 1+ \mathcal{O}_{\sigma}\l(\frac{N^{2\sigma} \Gamma^5 \gamma^3}{\sqrt{N \eta}}\r)\\
     - T_1(z)G^2 \Phi^N - z G^2 =& \mathbbm 1 + \mathcal{O}_{\sigma}\l(\frac{N^{2\sigma} \Gamma^5 \gamma^3}{\sqrt{N \eta}}\r) 
\end{aligned}
\end{equation}
where in the second line the error estimate holds componentwise. We can use the second equation to solve for $G^2$ as
$$
G^2=\frac{1}{-T_1 \Phi^N-z} \l(1+ \mathcal{O}_{\sigma}\l(\frac{N^{2\sigma} \Gamma^5 \gamma^3}{\sqrt{N \eta}}\r) \r)
$$
It remains to understand how the matrix $\frac{1}{-T_1 \Phi^N-z}$ affects the error estimate. This is relegated to the next subsection; see Proposition \ref{prop:errorprop}. Together with \eqref{eq:mT1T2decompose}, this proves Proposition \ref{prop:sc}.
\e{proof}

\subsection{Propagation of Error Bounds}
%

We set
\begin{equation}
    \|M\|_{\infty}:= \sup_{i,j}|M_{i,j}|.
\end{equation}
\begin{prop} \label{prop:errorprop}
Consider a solution to the system of equations 
\begin{equation} \label{eq:SystErr}
\begin{aligned}
    &{T_1(z) \left( -\frac{1}{N} \mathrm{Tr}(G^2(z) \Phi^N) -z \right)= 1 + e_1}\\
    & {- G^2(z) \Phi T_1(z) - z G^2(z) = I + E_2}, 
\end{aligned}    
\end{equation}
where $e_1$ is a number and $E_2$ is a matrix. Then
\begin{equation}
    {T_1(z) = -\frac{1}{T_1(z)} m_{\Phi}\l(-\frac{z}{T_1(z)}\r) + \tilde{e}_1},
\end{equation}
where $\tilde{e}_1$ satisfies,
\begin{equation}
    |\tilde{e}_1| \le \frac{1}{|z|}\l( |e_1| + (\log N)^6 \|E_2\|_{\infty} \r).
\end{equation}
\end{prop}

The proof uses spectral bounds on the finite Toeplitz matrix $\Phi^N$. These are inherited from the infinite Toeplitz matrix $\Phi$ as summarized in the following lemma.

\begin{lm}[Spectral estimates for finite banded Toeplitz matrix]\label{lm:PhiNest}
There exist $N$-independent constants $c,C>0$ such that
\beq\label{eq:PhiNspecbounds}
c\leq \Phi^N \leq C
\eeq
\end{lm}

\be{proof}[Proof of Lemma \ref{lm:PhiNest}]
Recall that $g_f$ is the Fourier symbol of the infinite Toeplitz matrix $\Phi$ and so
$$
\inf\mathrm{spec}\,\Phi =\inf_{x\in [0,1]}g_f(x)=g_{\min}>0,\qquad \sup\mathrm{spec}\,\Phi =\sup_{x\in [0,1]}g_f(x)\leq C \|f\|_{H^1}^2.
$$
It is a well-known fact that the finite restriction of an infinite Toeplitz matrix has spectrum lying inside the spectrum determined by the Fourier symbol of the infinite matrix \cite[Lemma 6]{GrayToeplitz}. In the present case, this implies that if we define the auxiliary Toeplitz matrix
\beq\label{eq:ANdefn}
A^N_{i,j}=
\phi(i-j),\qquad 1\leq i,j\leq N
\eeq
then
$$
\inf\mathrm{spec}\,A^N\geq \inf_{x\in [0,1]}g_f(x)=g_{\min}>0,\qquad \sup\mathrm{spec}\,A^N \leq \sup_{x\in [0,1]}g_f(x)\leq C \|f\|_{H^1}^2.
$$
Recall that the matrix $\Phi^N$ defined by \eqref{eq:PhiNdefn} is cut off at the band edges $|i-j|=(\log N)^6$. Using Definition \eqref{eq:phidefn} of $\phi$, the fact that $f\in C^2(\mathbb T)$ and the mean-value theorem, we see that 
$$
\|A^N-\Phi^N\|=o(1),\qquad \textnormal{as }N\to\infty
$$
and so Lemma \ref{lm:PhiNest} follows from Weyl's norm perturbation theorem.
\e{proof}

\begin{proof}[Proof of Proposition \ref{prop:errorprop}]
We abbreviate $T_1\equiv T_1(z)$. Solving the second identity for $G^2$ yields
$G^2= -(I + E_2)( \Phi T_1 +z)^{-1}$ and then the first identity gives
\begin{equation}
    T_1\left( \frac{1}{N} \mathrm{Tr}(\Phi^N(\Phi^N T_1 +z)^{-1}) -z\right) = 1+ e_1 - \frac{1}{N} \mathrm{Tr}( E_2 (\Phi^N T_1 +z)^{-1} \Phi^N).
\end{equation}

After some algebraic manipulation, we derive
\begin{equation}
    T_1 = - \frac{1}{T_1} m_{\Phi^N}\l(-\frac{z}{T_1}\r) +\frac{1}{z}\l(-e_1 +{\frac{1}{N} \mathrm{Tr}(E_2(\Phi^N T_1 +z)^{-1} \Phi^N})\r)
\end{equation}

Our goal is now to  control the second error term $\frac{1}{N} \mathrm{Tr}(E_2 (\Phi^N T_1 + z)^{-1} \Phi^N)$. By the cyclic property of the trace, we know that 
$$
\mathrm{Tr}(E_2(\Phi^N T_1 + z)^{-1} \Phi^N) ={ \mathrm{Tr}( \Phi^N E_2 (\Phi^N T_1 + z)^{-1} )}.
$$ Now, since $\Phi^N$ is a band matrix of band size $2(\log N)^6$ and only bounded entries, we see that $\| \Phi^N 
E_2 \|_{\infty} \le C(\log N)^6 \| E_2 \|_{\infty}. $
To estimate the entries of $(\Phi^N T_1 +z)^{-1}$, we need to use the fact that we are considering the inverse of a banded matrix. We apply the following result on the decay of entries of the inverse of banded matrices from \cite{Che2016}.

\begin{lm}[Inversion of band matrices \cite{Che2016}] \label{lem:BandInv}
Let $A$ be an invertible infinite or finite matrix, which is $W$-banded in the sense that $A(i,j)=0$ if $|i-j|\ge W$. We have 

$$|(A^{-1})_{i,j}| \le  2(2W+1) \kappa(A) \alpha^{(|i-j| -W)_{+}} ,$$
where 
$$
\kappa(A)= \|A \|\cdot \| A^{-1} \|,\qquad \alpha = \left(\frac{\kappa(A)-1}{\kappa(A) +1} \right)^{2/(2W+1)}.
$$

\end{lm}

We observe that $\kappa(T_1 \Phi^N +z)$ as defined in Lemma \ref{lem:BandInv}  is uniformly bounded independenly of $N$. 
Indeed, by Lemma \ref{lm:PhiNest}, we have 
$$
\| T_1 \Phi^N + z \| \ge c \omega
$$
because $\textnormal{Im}\langle v, (T_1 \Phi^N +z) v \rangle = \textnormal{Im}[T_1] \langle v, \Phi^N v \rangle + \textnormal{Im}[z] \langle v, v \rangle \ge (\textnormal{Im}[T_1] c + \textnormal{Im}[z]) \langle v , v \rangle.$ Similar arguments show that $\| T_1 \Phi^N +z \|$ is bounded from above.

The band width that we have to consider is $W=(\log N)^6$. Thus, we see that the constant $\alpha$ is given by,
\begin{equation}
    \alpha=\left(\frac{\kappa(T_1 \Phi^N +z)-1}{\kappa(T_1 \Phi^N +z)+1} \right)^{2/(2(\log N)^6+1)} = 1+ \mathcal{O}((\log N)^{-6}).
\end{equation}
We can then control individual entries of the matrix $\Phi^N E_2  (\Phi^N T_1 +z)^{-1}$, namely
\beq
\begin{aligned}
    \| \Phi^N E_2  (\Phi^N T_1 +z)^{-1}\|_{\infty} &\le 2(2(\log N)^6+1) \kappa(T_1 \Phi^N +z) \sum_{k=-N}^N \| \Phi^N E_2\|_{\infty} \alpha^{(|k| -(\log N)^6)_+}\\
    & \le C (\log N)^6 \| \Phi^N E_2\|_{\infty}\l((\log N)^6 + \frac{1}{1-\alpha}\r)\\
    & \le C{(\log N)^{18} }\| E_2\|_{\infty}.
\end{aligned}
\end{equation}
The bound on individual entries implies a bound on the normalized trace. This proves Proposition \ref{prop:errorprop}.
\end{proof}

\section{Derivation of the matrix self-consistent equation}\label{sec:matrissc}
In this section, we prove Lemma \ref{lm:scGF}. The following derivation is modeled after \cite{Che2016}. While we simplify some aspects, it is rather technical. Hence, we invite the reader to instead consider Appendix \ref{app:heuristic} for a heuristic derivation of  \eqref{eq:lmscGF} based on Gaussian integration by parts and the loop equation.

\subsection{Basic removal operations and bounds}
The following lemma can be seen as a generalization of the Schur complement formula. The proof is via the resolvent equation and can be found in \cite[Lemma 3.3]{Che2016}.

\begin{lm}\label{Lm:Resol}
Let $T,I, J$ be three subsets of $\{1,\ldots,2N\} $ satisfying $T \cap I, T \cap J = \emptyset$. Then, we have the following resolvent identities.
\begin{equation} \label{eq:Resolone}
    G_{I,J} = G^{(T)}_{I,J} + G_{I,T}(G_{T,T})^{-1} G_{T,J},
\end{equation}
\begin{equation} \label{eq:Resoltwo}
     G_{T,J} = - G_{T,T} H_{T, T^c}G^{(T)}_{T^c,J}.
\end{equation}

\end{lm}
As an immediate corollary, we can estimate the sizes of the Green's function after row removal in terms of the original Green's function (more precisely, in terms of the stochastic control parameters from Definition \ref{def:StCo}). This is used as a deterministic a priori estimate later.
\begin{cor} \label{col:StocBnd}
 Let $i$ be an integer in $\{1,\ldots, 2N\}$ and let
 $$
 I\subset [i - 2(\log N)^6, i + 2(\log N)^6]\cap \Z.
 $$
  Then 
\begin{equation} \label{eq:somebound}
    \vert G_{kl}^{(I)} \vert \le 8 (\log N)^6 \Gamma^2 \gamma,\qquad k,l\in \{1,\ldots,2N\}\setminus I.
\end{equation}
\end{cor}

\begin{proof}
Consider the identity of \eqref{eq:Resolone} with $I = \{k\}$, $J=\{l\}$ and $T=I$. We obtain
\begin{equation}
    G^{(I)}_{kl} = G_{kl} - G_{k, I} (G_{I,I})^{-1} G_{I, j}.
\end{equation}
We can treat the second term as an inner product and bound it by treating $G_{k,I}$ and $G_{I,j}$ as $4(\log N)^6$ dimensional vectors. That is
$$ 
\begin{aligned}
|G_{k, I} (G_{I,I})^{-1} G_{I, j}|
\le \|G_{k,I} \|_2\vert \| (G_{I,I})^{-1} \|_{2\to 2}\| G_{I,j} \|_2 
\le 4(\log N)^6 \Gamma^2 \gamma. 
\end{aligned}
$$
The first term $|G_{kl}|$ can be bounded by $\Gamma$, which is less than $4 (\log N)^6 \Gamma^2 \gamma$ since $\Gamma, \gamma \ge 1$ by definition.
\end{proof}

\subsection{Concentration Identities}
A key ingredient in the proof of \cite[Lemma 3.10]{Che2016} are concentration bounds for variables with finite range of dependence. Here we extend these estimates to the case of logarithmic range of dependence by grouping arguments similar to those in \cite{Che2016}.


\begin{lm} \label{lem:concen}
Let $a_1,\ldots,a_N$ be a family of dependent random variables that satisfy the property that $a_i$ and $a_j$ are independent whenever $|i-j|\geq \ceil{(\log N)^6}$

We fix a parameter $\sigma>0$, a power $p \ge p_0(\sigma)$ and assume the following moment estimate.
\begin{equation}
    \mathbb{E}[|a_i|^p] \le \frac{C^p}{\sqrt{N}^p}
\end{equation}
Then, there exists a universal constant $C>0$ such that for any collection of deterministic $A_{i}$ and $B_{ij}$ with $1\leq i,j\leq N$, with probability $1 -N^{-\sigma' p}$  and some $\sigma> \sigma'$, we have that 
\begin{align}\label{eq:concentration1}
    &\left \vert \sum_{i=1}^N A_i a_i - \mathbb{E}\left[\sum_{i=1}^N A_i a_i \right] \right\vert \le CN^{\sigma}  \left( \frac{\sup_i \vert A_i \vert}{\sqrt{N}} + \sqrt{\frac{1}{N}\sum_{i=1}^N |A_i|^2}\right),\\
    \label{eq:concentration2}
    & \left \vert \sum_{i,j=1}^N B_{ij} a_i a_j - \mathbb{E}\left [\sum_{i,j=1}^N B_{ij} a_i a_j\right] \right \vert \le C N^{\sigma}\left( \frac{\sup_{i,j}|B_{ij}|}{\sqrt{N}} + \sqrt{\frac{1}{N^2} \sum_{i,j=1}^N|B_{ij}|^2}\right).
\end{align}
\end{lm}
\begin{proof}
All sums are implicitly over $\{1,\ldots,N\}$. 
By dividing the integers modulo $\lceil( \log N)^6\rceil$, we see that each sum
\begin{equation}
    \mathcal{A}_k= \sum_{i = k \textnormal{ mod }\lceil( \log N)^6\rceil } A_i a_i - \mathbb{E}\l[\sum_{i = k \textnormal{ mod } \lceil( \log N)^6\rceil } A_i a_i \r],
\end{equation}
for $0\le k \le\lceil(\log N)^6\rceil-1$ is a sum of terms involving only independent $a_i$. We can then appeal to standard concentration estimates, see \cite[Theorem 7.7]{bYau}, to conclude that with probability $1- c_p N^{\sigma' p}$, we have
$$
\mathcal{A}_k \le N^{\sigma'}\left( \frac{\sup_i |A_i|}{\sqrt{N}} + \sqrt{\frac{1}{N}\sum_{i=1}^N |A_i|^2 }\right),
$$
 where $\sigma'<\sigma$. A union bound for all the sums $\mathcal{A}_k$ gives the desired estimate. We can absorb $(\log N)^6$ factors into $N^{\sigma}$ both in the bound \eqref{eq:concentration1}.

We come to the quadratic summation, where we have to do some more manipulation. First, we split
\begin{equation} \label{eq:Bsplit}
\begin{aligned}
    &\sum_{i,j} B_{i,j}a_{i} a_{j}-\mathbb E\left[\sum_{i,j} B_{i,j}a_{i} a_{j}\right]\\
     = &\sum_{ |i-j|>(\log N)^6} B_{i,j} a_i a_j + \sum_{|i-j|\le (\log N)^6}^N B_{i,j} a_i a_j
     -\mathbb E\l[ \sum_{|i-j|\le (\log N)^6}^N B_{i,j} a_i a_j\r]
    \end{aligned}
\end{equation}

Consider the first term on the right-hand side. It can be split further into a total of $(\log N)^{12}$ parts based on the moduli of $i$ and $j$ with $(\log N)^6$; call them
\begin{equation}
\begin{aligned}
    \mathcal{B}_{\al,\beta}= &\sum^N_{\substack{ |i-j| > (\log N)^6\\ i =\al \textnormal{ mod }\lceil( \log N)^6\rceil\\ j = \beta \textnormal{ mod }\lceil( \log N)^6\rceil}} B_{i,j} a_i a_j\end{aligned}
\end{equation}
Each $B_{\al,\beta}$ is a centered sum of independent random variables and so we can appeal to standard concentration estimates and a union bound with the logarithmic factor $(\log N)^{12}$ absorbed by the power function.

It remains to estimate the second and third terms on the right-hand side of \eqref{eq:Bsplit}. To this end, we split the sum into $4(\log N)^6$ parts as follows. Let $\al,\beta$ be integers with $1\leq \al\leq \lceil 2(\log N)^6 \rceil$ and  $-(\log N)^6\leq \beta\leq (\log N)^6$ and set
\begin{equation}
\begin{aligned}
    \mathcal{C}_{\al,\beta}=& \sum_{\substack{ i-j = \beta\\ i = \al  \textnormal{ mod } \lceil 2(\log N)^6 \rceil}} B_{i,j} a_i a_j- \mathbb{E} \left[ \sum_{\substack{ i-j = \beta\\ i = \al  \textnormal{ mod } \lceil 2(\log N)^6 \rceil}} B_{i,j} a_i a_j\right],
\end{aligned}
\end{equation}
Each $ \mathcal{C}_{\al,\beta}$ is a centered sum of independent random variables and there are $4(\log N)^{12}$ choices of $\al$ and $\beta$, so the result follows from standard concentration estimates and a union bound.
\end{proof}

We can use the previous lemma, the Ward Identity, and our stochastic control parameters to write out an explicit bound for the sums we consider upon applying resolvent identities. 
\begin{cor}\label{cor:conc}
Let $a_i$ be random variables satisfying the same conditions as in Lemma \ref{lem:concen}; in addition. 
Assume that all $A_i ,B_{ij}$ belong to the set $\{G_{ij}^{(T)}\}_{i,j}$ for some subset $T$ of 
$$
[k - 2(\log N)^6,  k+ 2(\log N)^6]\cap \Z,\textnormal{ with } k \in \{1,\ldots, 2N\}.
$$

 Then, for any $\sigma>0$ and $p \ge p_0(\sigma)\ge 2$, we have with probability at least $1 - N^{-\sigma' p}$, and some $\sigma> \sigma'$
\begin{equation} \label{eq:MatrConc}
\begin{aligned}
    & \left|\sum_i A_i a_i - \mathbb{E}\left[\sum_i A_i a_i\right] \right| \le C \frac{N^{\sigma} \Gamma^2 \gamma}{\sqrt{N \eta}}, \\
    & \left|\sum_i B_{ij} a_i a_j - \mathbb{E}\left[ \sum_{i,j} B_{i,j} a_i a_j \right] \right| \le  C \frac{N^\sigma \Gamma^2 \gamma}{\sqrt{N \eta}}.
\end{aligned}
\end{equation}
\end{cor}
\begin{proof}
We focus on the first estimate in \eqref{eq:MatrConc}.
By Lemma \ref{lem:concen}, we have with probability at least $1- N^{-\sigma' p}$ that
\begin{equation}
    \left|\sum _i A_i a_i - \mathbb{E}\left[ \sum_i A_i a_i \right] \right| \le  N^{\sigma'} \left( \frac{\sup_i |G_{ki}^{(T)}|}{\sqrt{N}} + \sqrt{\frac{1}{N} \sum_i |G_{ki}^{(T)}}|^2\right).
\end{equation}

By applying the Ward Identity, Lemma \ref{lem:Ward}, we can bound the sum $\frac{1}{N} \sum_i |G_{ki}^{(T)}|^2$ by $\frac{\textnormal{Im}[G_{kk}^{(T)}]}{N\eta}$. Now we apply Corollary \ref{col:StocBnd} to bound the Green function terms by $8(\log N)^6 \Gamma^2 \gamma$. The $(\log N)$ factors can be absorbed into $N^{\sigma}$. This gives us the desired first inequality of \eqref{eq:MatrConc}. 

The second inequality in \eqref{eq:MatrConc} follows from similar arguments; we omit the details.
\end{proof}

\subsection{Preliminary self-consistent equation}
As in the case of independent variables, the derivation of the self-consistent equation with dependence relies on extracting appropriate matrix components to which the concentration estimates can then be applied in a second step.

By definition, the Green's function satisfies
\begin{equation}\label{eq:GFdefnidentity}
    \sum_k G_{ik} H_{kj} - z G_{ij} = \delta_{ij}.
\end{equation}
We let $T$ be the set of entries that are correlated with $j$, i.e.,
$$
T=[j-(\log N)^6,j+(\log N)^6]\cap \Z.
$$

\begin{defn}
Fix a parameter $\sigma$ and an integer $p$ such that $p \ge \frac{100}{\sigma}$. For two sequences of random variables $a = a^{(N)}$ and $b= b^{(N)}$, we say that $a =\mathcal O_{\sigma,p}(b)$ if there exists universal constants $c$ and $N_0$ such that for $N> N_0$, we know that $a^{(N)} \le b^{(N)}$ with probability at least $1 - (\log N)^c N^{-\sigma p}$. We will drop the superscript $(N)$ when the context is clear.
\end{defn}

The following lemma specifies a preliminary self-consistent equation.

\begin{lm}\label{lm:prelimSC}
Let $\sigma$ and $p \ge p_0(\sigma) \ge 2$. Then
\begin{equation}
    -G_{i T} H_{T T^c} G_{T^c T^c}^{(T)} H_{T^c j} - z G_{i j} = \delta_{i j} + \mathcal O_{\sigma,p}\l(\frac{N^{\sigma} \Gamma^2 \gamma}{\sqrt{N \eta}}\r).
\end{equation}
\end{lm}

\begin{proof}
We recall the simple bounds $|H_{kj}| \le \frac{C}{\sqrt{N}}$ and $|G_{ik}|\le \Gamma$. 

Consider \eqref{eq:GFdefnidentity}. In the sum $\sum_{k} G_{ik} H_{kj}$, we bound the terms with $k \in T$ by $\frac{\Gamma}{\sqrt{N}}$ and obtain
$$
    G_{i T^c} H_{T^c j} - z G_{ij} = \delta_{ij} + \mathcal O\l(\frac{(\log N)^6 \Gamma}{\sqrt{N }}\r).
$$

We distinguish cases based on whether $i \in T$ or $i \not \in T$. 

If $i \in T$, we apply equation \eqref{eq:Resoltwo} and find
$$
    -G_{iT}H_{T T^c} G_{T^c T^c}^{(T)} H_{T^c j} - z G_{ij} = \delta_{ij} + \mathcal O_{\sigma,p}\l(\frac{(\log N)^6 \Gamma}{\sqrt{N}}\r). 
$$
Otherwise, if $i \not \in T$, we first apply equation \eqref{eq:Resolone} to obtain
$$
    \sum_{k\in T^c} G_{ik}^{(T)}H_{kj} + \sum_{k \in T^c}G_{i T}(G_{T,T})^{-1} G_{T,k} H_{kj} - z G_{ij}= \delta_{ij} + \mathcal O \l(\frac{(\log N)^6 \Gamma}{\sqrt{N}}\r).
$$
We can apply Corollary \ref{cor:conc} to the first term and \eqref{eq:Resoltwo} to the second term to obtain
$$
    G_{i T} H_{T T^c} G_{T^c T^c}^{(T)} H_{T^c j}- z G_{ij}= \delta_{ij} + \mathcal O_{\sigma,p} \l(\frac{N^\sigma \Gamma^2 \gamma}{\sqrt{N \eta}}\r)
$$
with probability at least $1 -N^{-\sigma p}$ as desired.
\end{proof}

\subsection{Conclusion}
Our goal at this point is to use our concentration estimates to replace $H_{T T^c} G_{T^c T^c} H_{T^c j}$ with its expectation. At this point, we have ensured that the term {$H_{T^c j}$} is completely independent of the term $G^{(T)}_{T^c T^c}$. However, we now need to guarantee that the term coming from $H_{T T^c} $ is independent of $G^{(T)}_{T^c T^c}$. Based on the first index $k \in T$ of $H_{T T^c}$, we need to apply a further resolvent expansion to remove terms in $G_{T^c T^c}^{(T)}$ that are correlated with $H_{T T^c}$. 
\begin{proof}[Proof of Lemma \ref{lm:scGF}]
Our starting point is Lemma \ref{lm:prelimSC}, i.e.,
$$
    -G_{i T} H_{T T^c} G_{T^c T^c}^{(T)} H_{T^c j} - z G_{i j} = \delta_{i j} + \mathcal O_{\sigma,p}\l(\frac{N^{\sigma} \Gamma^2 \gamma}{\sqrt{N \eta}}\r).
    $$

We let $S \subset T^c$ be the set of entries correlated with $T$ and let $U = T \cup S$.  We first split the summation over $T^c=U^c\cup S$ as,
$$
    H_{T T^c} G^{(T)}_{T^c T^c} H_{T^c j} = H_{T U^c} G^{(T)}_{U^c U^c} H_{U^c j} + H_{T U^c} G^{(T)}_{U^c S} H_{S j} + H_{T S} G^{(T)}_{S T^c} H_{T^c j}.
$$

\textit{Step 1.} We identify the leading term by applying the resolvent identity to replace the first $G^{(T)}_{U^c U^c}$ with $G^{(U)}_{U^c U^c}$ plus what we will show are error terms.
Observe from \eqref{eq:Resolone} that 
\begin{equation}
    G^{(T)}_{U^c U^c} = G^{(U)}_{U^c U^c} +G^{(T)}_{U^c S}(G^{(T)}_{SS})^{-1} G^{(T)}_{S U^c}.
\end{equation}

In particular, we have,
\begin{equation}\label{eq:decompT}
    H_{T T^c} G^{(T)}_{T^c T^c} H_{T^c j} = H_{T U^c} G^{(U)}_{U^c U^c} H_{U^c j} + e_{Tj},
\end{equation}
where an individual entry $e_{kj}$ of the error vector $e_{Tj}$ can be written as
\begin{equation}
\label{eq:ekj}
    e_{kj} = H_{k U^c} G^{(U)}_{U^c S} (G^{(T)}_{SS})^{-1} G^{(T)}_{S U^c} H_{U^c j} + H_{k U^c} G^{(T)}_{U^c S} H_{S j} + H_{k S} G^{(T)}_{S T^c} H_{T^c j}. 
\end{equation}

Since, by construction, the term $G^{(U)}_{U^c U^c}$ is independent from the terms appearing in $H_{T U^c}$ and $H_{U^c j}$, we are able to apply our concentration estimates as desired. 

Indeed, fix an entry $k \in T$. We apply the second concentration estimate in \eqref{eq:MatrConc} and \eqref{eq:Sigmadefn} to obtain
\begin{equation}
\begin{aligned}
    H_{k U^c} G^{(U)}_{U^c U^c} H_{U^c j} 
    =&\frac{1}{N}\sum_{l,m\in U^c}  G^{(U)}_{lm} \xi_{kljm}+ \mathcal O_{\sigma,p}\l(\frac{ N^{\sigma} \Gamma^2 \gamma}{\sqrt{N \eta}}\r).
    \end{aligned}
\end{equation}
We now have to replace $G^{(U)}_{lm}$ with $G_{lm}$ by reversing the resolvent expansion, i.e.,
\begin{equation}
    G^{(U)}_{lm} = G_{lm} - G_{lU}(G_{U,U})^{-1} G_{U m}.
\end{equation}

{We can apply the Cauchy-Shwarz inequality plus the Ward identity to bound}
\begin{equation}
    \left| \sum_{l ,m\in U^c}G_{lU} (G_{U,U})^{-1} G_{U m}\right|\le ||G_{\cdot U}||_2 ||(G_{U,U})^{-1}||_{2 \to 2} || || G_{U \cdot}||_{2} \le  \mathcal O\left(\frac{ \Gamma \gamma}{N \eta} \right)
\end{equation}

Combining these estimates, we have shown that
\begin{equation}
\begin{aligned}
    -\sum_{k,l,m} G_{i k} \xi_{kljm}G_{lm}- z G_{i j} 
    = \delta_{i j} +\sum_{k\in T} G_{ik}e_{kj}+ \mathcal O_{\sigma,p}\l(\frac{N^{\sigma} \Gamma^3 \gamma}{\sqrt{N \eta}}\r)
    \end{aligned}
\end{equation}

\textit{Step 2.}
It remains to estimate the error terms $e_{kj}$ from \eqref{eq:ekj}, or more precisely $\sum_{k\in T} G_{ik}e_{kj}$. 

For any fixed entry $s \in S$, we apply Corollary \ref{cor:conc} to $G^{(T)}_{S U^c} H_{U^c j}, H_{k U^c} G^{(T)}_{U^c S}$, and $G^{(T)}_{S T^c} H_{T^c j}$ to obtain
\begin{equation}
  G^{(T)}_{S U^c} H_{U^c j},\quad  H_{k U^c} G^{(T)}_{U^c S} ,\quad G^{(T)}_{S T^c} H_{T^c j} \le \mathcal O_{\sigma,p}\l(\frac{N^{\sigma}\Gamma^2 \gamma}{\sqrt{N \eta}}\r).
\end{equation}
For the remaining terms, e.g., $H_{k U^c}G^{(U)}_{U^c S}$, we can no longer apply the concentration estimate. However, observe from the resolvent expansion \eqref{eq:Resoltwo}, $T^c=U^c\cup S$  and equation \eqref{eq:somebound}
 that,
\begin{equation}
    H_{T U^c} G^{(T)}_{U^c S} =- (G_{T T})^{-1} G_{T S}+\mathcal O \l((\log N)^{12}\frac{\Gamma^2 \gamma}{\sqrt{N}}\r) 
    =\mathcal O \l((\log N)^{12}\Gamma^2\gamma\r)
\end{equation}
We conclude that 
\beq
\sum_{k\in T} G_{ik}e_{kj}= \mathcal O_{\sigma,p}\l(\frac{N^{\sigma} \Gamma^3 \gamma^2}{\sqrt{N \eta}}\r),
\eeq
 as desired. This completes Step 2 and proves Lemma \ref{lm:scGF}.
\end{proof}


\section{Analysis of the self-consistent equation}\label{sec:analysis}

In this section, we study the scalar self-consistent equation that arises from Proposition \ref{prop:sc}.

\subsection{Auxiliary self-consistent equation}
It is convenient to focus on the following auxiliary scalar self-consistent equation that arises from Proposition \ref{prop:sc} by dropping the error term, namely
\beq\label{eq:scanalyze}
-\frac{1}{M_N(z)} m_{\Phi^N}\l(-\frac{z}{M_N(z)}\r) =  M_N(z).
\eeq
Note that this differs slightly from the limiting self-consistent equation \eqref{eq:sc} satisfied by $m_\infty(z)$. Indeed, \eqref{eq:scanalyze} still features the $N\times N$ Toeplitz matrix $\Phi^N$ while \eqref{eq:sc} contains the infinite Toeplitz matrix $\Phi$ instead. 

First, we prove existence, uniqueness, and stability of solutions to \eqref{eq:scanalyze}. Second, in Lemma \ref{lem:Self-Consist3}, we show that $M_N$  is close the solution of the limiting equation $m_\infty$ for large $N$. Finally, we study in detail the propagation of error estimates on the Green's function into the self-consistent equation and show that we can use solutions of \eqref{eq:scanalyze} to approximate the Green's function.

\subsection{Existence and uniqueness of solutions}

\begin{lm}[Existence] \label{lem:exist}
There exists a solution ${M_N(z)}: \mathbb{C}^+ \to \mathbb{C}^+$ to the equation \eqref{eq:scanalyze}.
\end{lm}

We emphasize that { $M_N(z)$ }denotes a scalar quantity. The proof uses Brouwer's fixed point theorem. 

\begin{proof}
Fix $z\in\C^+$.
We define the function $F(w)= - \frac{1}{w} m_{\Phi^N}(-\frac{z}{w})$ that appears on the left-hand side of \eqref{eq:scanalyze}. We can rewrite $F(w)$ as
\begin{equation}
    F(w)= - \int_{\mathbb{R}} \frac{ 1}{w x +z}\d\rho_{\Phi^N}(x)
\end{equation}
where $\d\rho_{\Phi^N}$ is the empirical spectral distribution of $\Phi^N$. 

We consider the compact, convex domain
\begin{equation}
    \mathcal{D}_z:=\left\{ w \in \mathbb{C}^+\;:\;|w| \le \frac{1}{\textnormal{Im}[z]},\quad \textnormal{Im}[w]  \ge \frac{\textnormal{Im[z]}}{\left( \frac{C}{\textnormal{Im}[z]} +|z| \right)^2}\right \}
\end{equation}
or an appropriate constant $C>0$ to be determined. Below we show that $F:\mathcal{D}_z\to\mathcal{D}_z$. Then Brouwer's fixed point theorem implies that $F$ has a fixed point. This fixed point is then the desired solution $M_N(z)\in \mathcal D_z$ which we note has positive imaginary part.

It remains to establish  $F:\mathcal{D}_z\to\mathcal{D}_z$. Let $w\in\mathcal D_z$.  Since $z,w$ have positive imaginary part { and $x$ is non-negative},
$|w x + z| \ge \textnormal{Im}[z]$ and so $|F(w)|\leq \frac{1}{\mathrm{Im}\,[z]}\rho_{\Phi^N}(\R)=\frac{1}{\mathrm{Im}\,[z]}$. Next, we consider the behavior of the imaginary part,
\begin{equation}\label{eq:Frep}
    \textnormal{Im}[F(w)] = \int_{\mathbb{R}} \frac{\textnormal{Im}[w] x  + \textnormal{Im}[z] }{|x w + z|^2} \d \rho_{\Phi^N}(x).
\end{equation}

From \eqref{eq:PhiNspecbounds}, we conclude that $\d\rho_{\Phi^N}$ is supported on $[g_{\min},C]\subset \R_+$. Using that $|w| \le \frac{1}{\textnormal{Im}[z]}$, we have on the support of $\d\rho_{\Phi^N}$ that  $\frac{1}{|w x + z|^2} \ge \frac{1}{\left(\frac{C}{\textnormal{Im[z]}} + |z| \right)^2}$. Hence
\begin{equation}
    \textnormal{Im}[F(w)] \ge \frac{\textnormal{Im}[z]}{\left(\frac{C}{\textnormal{Im[z]}} + |z| \right)^2}
\end{equation}
as desired.
\end{proof}

Since Brouwer's fixed point theorem does not imply uniqueness, we have to prove it by hand using the structure of the self-consistent equation.

\begin{lm}[Uniqueness] \label{lem:uniq}
Let $z\in\C^+$. Then there is a unique solution $M_N(z)$ to \eqref{eq:scanalyze} with positive imaginary part.
\end{lm}

\begin{proof}
Fix $z\in\C^+$. We first show that any solution $M_N(z)$ of \eqref{eq:scanalyze} with positive imaginary part must satisfy 
\begin{equation} \label{eq:Imagprtbnd}
    1 - \int_{\mathbb{R}} \frac{x}{|M_N(z)x + z|^2}  \d\rho_{\Phi^N}(x) > 0.
\end{equation}
Indeed, taking imaginary parts of both sides of \eqref{eq:scanalyze} in the representation \eqref{eq:Frep}, we obtain
\begin{equation} 
  \textnormal{Im}[M_N(z)] =  \textnormal{Im}[M_N] \int_{\mathbb{R}} \frac{x  }{|M_N(z) x + z|^2}   \d\rho_{\Phi^N}(x) 
    + \textnormal{Im}[z] \int_{\mathbb{R}} \frac{1}{|{M_N(z)}x+ z|^2}  \d\rho_{\Phi^N}(x) 
\end{equation}
from which \eqref{eq:Imagprtbnd} follows by rearranging.

Now suppose that we have two solutions $M$ and $\tilde M$ to \eqref{eq:scanalyze}. Then we rewrite $M-F(M)-(\tilde M-\tilde F(M))=0$ using \eqref{eq:Frep} as 
\begin{equation} \label{eq:uniq}
\begin{aligned}
   & M- \tilde{M} - (M- \tilde{M}) \int_{\mathbb{R}} \frac{x }{(M x +z)(\tilde{M}x+z)}\d\rho_{\Phi^N}(x)=0\\
   & (M - \tilde{M})\left[ 1 - \int_{\mathbb{R}} \frac{x }{(Mx+z)(\tilde{M}x +z)}\d\rho_{\Phi^N}(x)\right]=0.
\end{aligned}
\end{equation}
Now, observe that by \eqref{eq:Imagprtbnd},
\begin{equation}
    \left \vert\int_{\mathbb{R}} \frac{x}{(Tx+z)(\tilde{T}x +z)}\d\rho_{\Phi^N}(x) \right \vert \le \left(\int_{\mathbb{R}} \frac{x }{|Tx +z|^2}\d\rho_{\Phi^N}(x) \right)^{1/2} \left( \int_{\mathbb{R}} \frac{x }{|\tilde{T}x+z|^2}\d\rho_{\Phi^N}(x)\right)^{1/2}< 1.
\end{equation}
This proves $M=\tilde M$ as desired.
\end{proof}

\subsection{Stability estimates}
The estimates used in the proof of the last lemma can be refined show stability of solutions to \eqref{eq:scanalyze}. Recall that $M_N(z)$ is the unique solution with positive imaginary part to 
\begin{equation} \label{eq:finaleq}
    M_N(z) = -\frac{1}{M_N(z)} m_{\Phi^N} \left( -\frac{z}{M_N(z)} \right).
\end{equation}

\begin{lm}[Stability] \label{lem:Stab}

Let $\Phi\in \{\Phi^N,\Phi\}$ and let $T_0$ solve
\begin{equation}
    T_0 = - \frac{1}{T_0} m_{\Phi}\left(-\frac{z}{T_0}\right).
\end{equation}
Moreover, let $T$ be a solution to the approximate self-consistent equation,
\begin{equation}
    T = - \frac{1}{T} m_{\Phi}\left(-\frac{z}{T}\right) + \mathcal{E},
\end{equation}
for some $\mathcal E\in \C$. 
Fix a parameter $\omega>0$ and assume we are at a point $z$ such that { $\textnormal{Im}[T_0(z)] \ge \omega$}. Then, there exist constants $\epsilon_{\omega}>0$ and $C_{\omega}$ depending only on $\omega$ such that if $|\mathcal{E}| \le \epsilon_{\omega}$ and $|T - T_0| \le \epsilon_{\omega}$, then we have,
\begin{equation}
    |T-T_0| \le C_{\omega} |\mathcal{E}|.
\end{equation}

\end{lm}

\begin{rmk}
The condition that { $\textnormal{Im}[T_0] \ge \omega$} means that we consider values of $z$ whose real part corresponds to the bulk of the spectrum. 
\end{rmk}

\begin{proof}
Following the manipulations of \eqref{eq:uniq} in the previous Lemma \ref{lem:uniq}, we can assert that 
\begin{equation}
    (T- T_0)\left(1- \int_{\mathbb{R}} \frac{x}{(Tx + z)(T_0 x +z)}\d\rho_{\Phi}(x)\right) = \mathcal{E}.
\end{equation}
We will write $Tx+z = T_0x +z + (T- T_0) x$ and expand the denominator. By Lemma \ref{lm:PhiNest}, we see that $|T_0 x +z| \ge  c \textnormal{Im}[T_0] \geq c\omega$ on the support of $\d\rho_{\Phi}$. 
Hence
\begin{equation}
\begin{aligned}
    &\left \vert  \int_{\mathbb{R}} \frac{x}{(Tx+z)(T_0x +z)}\d\rho_{\Phi}(x) -\int_{\mathbb{R}} \frac{x}{(T_0 x+z)^2}\d\rho_{\Phi}(x)\right \vert\\
     \le&  \int_{\mathbb{R}} \frac{x^2|T- T_0|\textnormal{d}x}{|T_0x+z|^2\|T_0x +z| - |(T-T_0)x\|} \d\rho_{\Phi}(x)\\
    \le&  \frac{C^2 \epsilon_\omega}{(c \omega )^2{ |c\omega-\epsilon_{\omega}|}}
\end{aligned}
\end{equation}
where $C$ is an upper bound on the support of $\d\rho_{\Phi}$. Writing $K(z)$ for this difference, we have shown that, for sufficiently small $\eps_\omega$, \begin{equation} \label{eq:error}
    (T - T_0)\left( 1 - \int_{\mathbb{R}} \frac{x }{(T_0x +z)^2}\d\rho_{{\Phi}}(x) - K(z) \right) = \mathcal{E}.
\end{equation}
where $|K(z)|\leq \frac{C^2 \epsilon_{\omega}}{{c^2 \omega^2(c\omega- \epsilon_\omega)}}$.

We claim that there exists $c_\omega$ such that $ 1- \int_{\mathbb{R}} \frac{x \rho_{{\Phi}}(x) \textnormal{d}x}{(T_0x+z)^2}>c_{\omega}$. The claim then follows by making $\eps_\omega$ sufficiently small, which we point out also makes $|K(z)|$ small.

Recall that from looking at the imaginary parts of the self-consistent equation, we have the relation,
\begin{equation}
    \textnormal{Im}[T_0] \left(1 - \int_{\mathbb{R}} \frac{x}{|T_0x +z|^2} \d\rho_{\Phi}(x)\right) = \textnormal{Im}[z] \int_{\mathbb{R}} \frac{1 }{|T_0 x+z|^2}\d\rho_{\Phi}(x),
\end{equation}
which implies that $\int_{\mathbb{R}} \frac{x}{|T_0x +z|^2} \d\rho_{\Phi}(x)< 1$. By using the fact that $T_0$ has strictly positive imaginary part $> \omega$, we can deduce that there is an even bigger gap between $1$ and $\int_{\mathbb{R}} \frac{x  }{(T_0x +z)^2}\d\rho_{\Phi}(x)$.

Notice that in the support of $\d\rho_{{\Phi}}$ we know that $\textnormal{Im}[T_0x +z] \ge c \omega$ and $|T_0 x +z| \le C(z)$ for some constant $C(z)>0$. When writing $T_0 x +z = |T_0x+z| e^{i \theta(x,z)}$ in polar coordinates, we find that $\sin \theta > \frac{c \omega}{C(z)}$.  In particular, there is some $\theta_\omega$ such that $\theta_\omega \le \theta (x,z)\le \pi - \theta_\omega$. We have
\begin{equation}
\frac{1}{(T_0x+z)^2} = \frac{e^{-2 i \theta(x,z)}}{|T_0x +z|^2}.
\end{equation}
By integrating this over $x$, we obtain
\begin{equation}
    \int_{\mathbb{R}} \frac{x } {(T_0 x+z)^2} \d\rho_{\Phi}(x)= \int_{\mathbb{R}}e^{-2i\theta(x,z)} \frac{x}{|T_0 x +z|^2}\d\rho_{\Phi}(x) 
\end{equation}
Recalling that $\int_{\mathbb{R}} \frac{x \d\rho_{{\Phi}}(x)}{|T_0 x +z|^2} <1$ and $2 \theta_\omega \le 2\Theta \le 2 \pi - 2 \theta_\omega$, this relation implies that there is a strictly positive gap between
$1$ and $\int_{\mathbb{R}} \frac{x \rho_{{\Phi}}(x) \textnormal{d}x} {(T_0 x+z)^2}$ that depends only on $\omega$.

Thus, for sufficiently small $\epsilon_{\omega}$, we know that there exists some constant $C_{\omega}^{-1}$
\begin{equation}
   \left \vert 1 - \int_{\mathbb{R}} \frac{x }{(T_0x +z)^2} \d\rho_{{\Phi}}(x) - K \right\vert> C_{\omega}^{-1}.
\end{equation}
Considering \eqref{eq:error}, this implies that $|T - T_0| \le C_{\omega}|\mathcal{E}|$.
\end{proof}

\subsection{Comparison to the limiting self-consistent equation}
Recall that $m_\infty(z)$ is the unique solution of the limiting self-consistent equation \eqref{eq:sc}. Here we prove that it is close to $M_N(z)$, the solution of the auxiliary self-consistent equation \eqref{eq:scanalyze} that was analyzed in the preceding subsections

We can combine our discussion in the following lemma,
\begin{lm} \label{lem:Self-Consist3}
Assume that $\textnormal{Im}[M_n(z)] \ge \omega$.
Then
\begin{equation}
    |M_N(z) -m_{\infty}(z)| \le C \frac{{(\log N)^{12}}}{N}.
    \end{equation}
\end{lm}

The proof idea is that, while the $N \times N$ matrix $\Phi^N$ does not have an explicitly computable eigenvalue distribution, a logarithmic-rank perturbation will. Since $\Phi^N$ was a band matrix with band size $(\log N)^6$, we can extend the band by wrapping around the other side of matrix, as if it were a torus (such matrices are called circulant matrices) and then explicitly compute the spectrum.

\be{proof}
We define the circulant matrix
\begin{equation}
    (\tilde{\Phi}^N)_{i,j} = \phi(i-j),\qquad \textnormal{if }|i-j \mod N|\leq (\log N)^6.
\end{equation}
Note that $\tilde{\Phi}^N$ is a Hermitian matrix and a rank $2(\log N)^6$ perturbation of $\Phi^N$. By taking the finite Fourier series, we see that $\tilde{\Phi}^N$ has eigenvalues explicitly given by 
$$
\tilde{e}_k:= \sum_{j=-(\log N)^6} ^{(\log N)^6} \phi(j) e^{2\pi i \frac{k}{N}j}
$$
 where $k$ can take any integer value between $0$ and $N-1$. We order these eigenvalues as
\begin{equation} \label{eq:ordpert}
    \tilde{e}_{\pi(1)} \le\tilde{e}_{\pi(2)} \le \tilde{e}_{\pi(3)} \le \ldots \le \tilde{e}_{\pi(n)}
\end{equation}
for an appropriate permutation $\pi\in S_N$.

Observe the following inequality for some constant $C$,
\begin{equation}
    |\tilde{e}_{k} - \tilde{e}_{k+1}| \le \sum_{j=-(\log N)^6}^{(\log N)^6} |\phi(j)| |e^{2\pi i \frac{j}{N}} -1| \le C \sum_{j=-(\log N)^6}^{(\log N)^6} |\phi(j)| \frac{(\log N)^6}{N}.
\end{equation}
This inequality implies that $|\tilde{e}_{\pi(k)} - \tilde{e}_{\pi(k+1)}| \le C\frac{(\log N)^6}{N}$ for a potentially different constant $C$.

If we now denote the eigenvalues of $\Phi^N$ by
\begin{equation} \label{eq:ordeig}
    e_{\pi(1)} \le e_{\pi(2)} \le e_{\pi(3)} \le \ldots \le e_{\pi(n)},
\end{equation}
then we see that by interlacing (since the difference between $\Phi^N$ and $\tilde{\Phi}^N$ is a rank $2(\log N)^6$ matrix), we have 
\begin{equation} \label{eq:inter}
     \tilde{e}_{\pi(j - 2(\log N)^6)}\le e_{\pi(j)} \le \tilde{e}_{\pi(j +2(\log N)^6)}
\end{equation}
Moreover, this implies that for $ 2(\log N)^6 < j < N - 2(\log N)^6$. we have,
\begin{equation}
    |e_{\pi(j)} - \tilde{e}_{\pi(j)}| \le \frac{C (\log N)^{12}}{N}.
\end{equation}

{ I will assume in what follows that you want to show that $M_N(z)$ is an approximate solution to the limiting self-consistent equation \eqref{eq:sc} instead of \eqref{eq:scanalyze}, as it was previously}
The idea is now to use these estimates to show that $T\equiv M_N(z)$ is an approximate solution to { \eqref{eq:sc} }and conclude by stability. We start from

\begin{equation}
\begin{aligned}
  \l|T-\frac{1}{T}{m_{\Phi}}\l(-\frac{z}{T}\r)\r|
  =&    \left|\frac{1}{T} m_{\Phi^N}\l(-\frac{z}{T}\r) - \frac{1}{T} { m_{\Phi}}\l(-\frac{z}{T}\r) \right|\\
     =& \left|\sum_{j=1}^N \left[\frac{1}{N} \frac{1}{T e_{\pi(j)} +z } - \int_{\frac{\pi(j)}{N}}^{\frac{\pi(j)+1}{N}}   \frac{\textnormal{d}x}{T \sum_{j=-\infty}^{\infty} \phi(j) e^{2\pi i j x} + z}\right]\right|\\
     \le& \sum_{j=1}^N \int_{\frac{\pi(j)}{N}}^{\frac{\pi(j)+1}{N}} \left| - \frac{1}{T e_{\pi(j)} +z} + \frac{1}{T \sum_{j=-\infty}^{\infty} \phi(j) e^{2\pi i j x} +z}\right| \textnormal{d}x\\
     \le&  C \sum_{j=1}^N \int_{\frac{\pi(j)}{N}}^{\frac{\pi(j)+1}{N}} \l|e_{\pi(j)} - \sum_{j=-\infty}^{\infty} \phi(j) e^{2\pi i j x} \r|.
\end{aligned}
\end{equation}
We estimate the last term by using that, for $(\log N)^6 < j < N - (\log N)^6$ and $x$ as above,
\begin{equation}
\begin{aligned}
    &\l|e_{\pi(j)} - \sum_{j=-\infty}^{\infty} \phi(j) e^{2\pi i j x}\r|\\
     \le& |e_{\pi(j)} - \tilde{e}_{\pi(j)}| + \l|\tilde{e}_{\pi(j)} - \sum_{j=-\infty}^{\infty} \phi(j) e^{2\pi i j x} \r| \\
    \le& |e_{\pi(j)} - \tilde{e}_{\pi(j)}| + \sum_{j=-(\log N)^6}^{(\log N)^6} |\phi(j)| |e^{2\pi i j x} - e^{2 \pi i j \frac{\pi(j)}{N}}| +\sum_{j \in \mathbb{Z}, j \not \in [-(\log N)^6, (\log N)^6]} |\phi(j)|\\
    \le& C \frac{{(\log N)^{12}}}{N} + \frac{C(\log N)^6}{N} \sum_{j=-(\log N)^6}^{(\log N)^6}|\phi(j)| + \sum_{j \in \mathbb{Z}, j \not \in [-(\log N)^6, (\log N)^6]} |\phi(j)|.
\end{aligned}
\end{equation}
The other eigenvalues with $j\leq (\log N)^6$ or  $j\geq N - (\log N)^6$ are bounded by uniform a priori bounds. It follows that
\begin{equation} 
    T = - \frac{1}{T} { m_{\Phi}}\l(-\frac{z}{T}\r) + \mathcal{E},
\end{equation}
with the error estimate
\begin{equation}
    |\mathcal{E}| \le C \frac{{(\log N)^{12}}}{N} + C \sum_{j \in \mathbb{Z}, j \not \in [-(\log N)^6, (\log N)^6]}|\phi(j)|.
\end{equation}
By applying our stability result, Lemma \ref{lem:Stab}, we see that for $N$ sufficiently large,
\begin{equation}
    |M_N(z) -m_{\infty}(z)| \le C \frac{(\log N)^6}{N} + C \sum_{j \in \mathbb{Z}, j \not \in [-(\log N)^6, (\log N)^6]}|\phi(j)|.
\end{equation}
The claim now follows by estimating the second term, using Definition \eqref{eq:phidefn} of $\phi$, the fact that $f\in C^2(\mathbb T)$ and the mean-value theorem.
\end{proof}

\subsection{Application to Green's function estimates}
In this subsection, we show that we can recover the Green's function block $G^2$ from solutions $m_\infty(z)$ to the limiting scalar self-consistent equation \eqref{eq:sc}.  { The exact same method will give us bounds on the entries of $G^3$, $G^4$, and the off-diagonal entries of $G^1$,}

\begin{lm} \label{lem:Self-Consist2}
Let $(T^1,G^2)$ be a solution to the system of equations \eqref{eq:SystErr}. Assume in addition to this that $|T^1 - m_\infty(z)| \le \mathcal{E}$ where $\mathcal{E}=o(1)$ as $N\to\infty$.  { Also recall the other matrices $G^1$,$G^3$ and $G^4$ from \eqref{eq:equationset}.  In this proof we will let $T_0(z) \equiv m_{\infty}(z)$, the solution to the limiting self-consistent equation \eqref{eq:sc}.}
Then
 \begin{equation}
\begin{aligned}
&\| G^1- T_0(z) \mathbbm{1}\| \le C (\log N)^6\|E_1\|_{\infty} + C (\log N)^{12} \mathcal{E}.\\
  &  \| G^2 + (\Phi^N { T_0(z)} +z)^{-1}\|_{\infty} \le C (\log N)^6\|E_2\|_{\infty} + C (\log N)^{12} \mathcal{E}.\\
&	\| G^3\|_{\infty} \le C (\log N)^6\|E_3\|_{\infty} + C (\log N)^{12} \mathcal{E}.\\
&	\|G^4\|_{\infty} \le  C (\log N)^6\|E_4\|_{\infty} + C (\log N)^{12} \mathcal{E}.\\
\end{aligned}
\end{equation}

\end{lm}

\begin{proof}

{ We will drop the argument $z$ from $T_1(z)$ and $T_0(z)$ when the context is obvious. It suffices to prove the result for $G^2$. All other blocks are similar.}

It follows from \eqref{eq:SystErr} that
 \begin{equation}
     G^2 = -(\Phi^N T_1 + z)^{-1} - (\Phi^N T_1 + z)^{-1} E_2.
 \end{equation}
 
We {recall }$T_0\equiv m_\infty(z)$. We expand the first resolvent using the resolvent identity and find
 \begin{equation}
     (\Phi^N T_1 +z)^{-1} = (\Phi^N T_0 +z)^{-1} - (\Phi^N T_0 +z)^{-1}(T_1 - T_0) (\Phi^N T_1 +z)^{-1}. 
 \end{equation}
We bound the entries of $(\Phi^N T_0 +z)^{-1}$ and $(\Phi^N T_1 +z)^{-1}$ by Lemma \ref{lem:BandInv} since they are inverses of a band matrix.  The function $\kappa$ is bounded as in the proof of Lemma \ref{lem:BandInv}. Hence , we  obtain for every $1\leq i,j\leq N$\begin{equation}
 \begin{aligned}
     |((\Phi^N T_0+z)^{-1}(\Phi^N T_1 +z)^{-1})_{ij}| &\le \sum_{k=1}^N |[(\Phi^N T_0 +z)^{-1}]_{ik}|
     |[(\Phi^N T_1+z)^{-1}]_{kj}|\\
     & \le C \sum_{k=1}^N \alpha_0^{(|i-k|-(\log N)^6)_{+}} (\alpha^1)^{(|j-k|-(\log N)^6)_{+}}\\
     & \le C \sum_{k=-\infty}^{\infty} \alpha_0^{(|i-k| - (\log N)^6)_{+}} \sum_{k=-\infty}^{\infty} (\alpha^1)^{|j-k|- (\log N)^6}\\
     & \le C \l((\log N)^6 + \frac{1}{1-\alpha_0}\r)\l((\log N)^6 + \frac{1}{1-\alpha^1}\r) { (\log N)^{0}},
 \end{aligned}
 \end{equation}
 Here, $\alpha_0$ and $\alpha^1$ are the corresponding constants from the application of Lemma \ref{lem:BandInv} to the matrices $(\Phi^N T_0 +z)$ and $(\Phi^N T_1 +z)$, respectively. The assumption that $\textnormal{Im}[T_0] \ge \omega$ clearly shows that $\| \Phi^N T_0 +z\|$ is bounded above and below and, thus, we can assert that $\alpha_0 = 1+ \textnormal{O}((\log N)^{-6})$ and the result of infinite summation $(1-\alpha_0)^{-1}= \textnormal{O}((\log N)^6)$.  
 
 In addition to this, we have already assumed that $|T_0- T_1|= \mathcal{E} =o(1)$, so clearly we have the same upper and lower bounds on $\|\Phi^N T_1 +z\|$, and we may still assert that $\alpha^1 = 1+ \textnormal{O}((\log N)^6)$.  This shows that we have the following estimate,
 \begin{equation}
     \| G^2 + (\Phi^N T_0 +z)^{-1}\|_{\infty} \le  C(\log N)^6\|E_2\|_{\infty} + C (\log N)^{12} \mathcal{E}.
 \end{equation}
\end{proof}

\section{Establishing the Local Law} \label{sec:locallaw}

Once we have derived the error estimates for our self-consistent equation along with the stability estimates, we can prove a local law via a standard continuity approach. 

\subsection{The Global Law}

The first goal is to establish the local law at large scales, e.g. a global law. We will establish the following theorem,
\begin{thm}
Let $\mathcal{M}$ be the exact solution to the matrix  of self-consistent equations \eqref{eq:firsteq}, let  { $m_{\infty}(z)$} be the exact solution to the infinite self-consistent equation
\eqref{eq:sc}, and let $G$ be the Green's function of the matrix {$Y$} as in section \ref{sec:self-consist}. Let $\mathcal{D}$ be a compact subset of $\mathbb{C}^+$. Then, for $N_{\mathcal{D},\nu,p}$ sufficiently large depending on $\mathcal{D},\nu$ and $p$, we can establish the following bound for $N \ge N_{\mathcal{D},\nu,p}$.
\begin{equation}\label{eq:global}
 \begin{aligned}
 &\mathbb{P}\left(\sup_{z \in \mathcal{D},i,j} |G_{ij} -\mathcal{ M}_{ij}| \ge \frac{N^{\nu}}{\sqrt{N\eta}}\right) \le N^{-\nu p},\\
 & \mathbb{P}\left(\sup_{z \in \mathcal{D}} |\frac{1}{2N} \mathrm{Tr} G(z) - {m_{\infty}(z)}| \ge \frac{N^{\nu}}{\sqrt{N\eta}}\right) \le N^{-\nu p}.
\end{aligned}
\end{equation}
\end{thm}
\begin{proof}
We see from Lemma \ref{lm:scGF} that we know that the matrix $G$ satisfies the self-consistent equation up to error of order given by $\frac{N^{2\sigma}\Gamma^5 \gamma^3}{\sqrt{N \eta}}$. From the results of {  Lemmas \ref{lem:Stab}, \ref{lem:Self-Consist2},  and \ref{lem:Self-Consist3}}, we know that we can derive the desired result \eqref{eq:global} as long as we know that $\Gamma, \gamma \lesssim 1$ for $z \in \mathcal{D}$. (Note that since $G$ and $\mathcal{M}$ are Lipschitz, we can derive the high probability bounds on a discrete grid that is of polynomial size and extend to the entire set $\mathcal{D}$ without too much loss in probability.) 

We have the deterministic bound that $\Gamma \le \frac{1}{\eta} $, which will be bounded by $1$ in our compact region $\mathcal{D}$.   It suffices to establish bounds on $\gamma$. By using the Schur complement formula, we know that
\begin{equation}
    (G^{(J)}_{I,I})^{-1}= H_{I,I} - z + H_{I, K^c} G_{K^c,K^c}^{(K)}H_{K^c, I}, 
\end{equation}
for $K= I \cup J$, which will be a subset of $[i - 2(\log N)^6, i + 2(\log N)^6]$.

We can attempt to estimate the operator norm of the right hand side. First, observe that we can bound the operator norm of $H_{I,I}$ by the Frobenius norm,
\begin{equation}
    \|H_{I,I}\| \le \sqrt{\sum_{a,b \in I}|H_{ab}|^2},
\end{equation}
which can be bounded by $O(\frac{N^{\sigma}}{\sqrt{N \eta}})$. This only uses the fact that $|H_{ab}|=\frac{O(1)}{\sqrt{N}}$ and we have at most $O((\log N)^{12})$ terms in the sum.

Finally, we can estimate,
\begin{equation}
    \| H_{I, K^c} G^{(K)}_{K^c K^c} H_{K^c I}\| \le \|H_{I,K^c}\| \|G^{(K)}_{K^c K^c}\| \|H_{K^c, I}\| \le \|G^{(K)}_{K^c,K^c}\| \sum_{i\in I,k \in K^c} |H_{ik}|^2.
\end{equation}

We know that $\|G^{(K)}_{K^c,K^c}\| \le \eta^{-1}$ and $\sum_{i\in I,k \in K^c} |H_{ik}|^2 \le O((\log N)^6)$. Combining our previous estimates by the triangle inequality, this shows that $\gamma \le O(N^{\sigma})$ in $\mathcal{D}$. This completes the proof.

\end{proof}

\subsection{Proving the Local Law}

As we have seen in the previous proof, it was necessary to get a bound on $\gamma$ and $\Gamma$. We will establish the local analogue first.
\begin{lm} \label{lem:BoundGamma}
Assume that 
\begin{equation} \label{eq:stabcond}
{\|G -\mathcal{ M}\|_{\infty}  \le N^{-\varepsilon}} ,
\end{equation}
for some parameter $\varepsilon >0$. Then, we have that
\begin{equation}
    \Gamma, \gamma \le O(1).
\end{equation}
\end{lm}
\begin{proof}
    The explicit form of the solution of the self-consistent equation ensures that { $\|\mathcal{M}\|_{\infty} = O(1)$}. Clearly, the condition \eqref{eq:stabcond} would imply that $\Gamma = O(1)$ immediately. We need to do more work to establish the same result for $\gamma$. 
    
    As is usual, we let $J$ and $I$ be non-intersecting subsets of $[ i - 2(\log N)^6,\ldots, i + 2(\log N)^6]$ for some integer $i$.
    
    Recall the resolvent identity,
    \begin{equation}
        G_{J,J}^{(I)} = G_{J,J} -G_{J,I}(G_{I,I})^{-1} G_{I,J}.
    \end{equation}
    We will compare this to the corresponding quantity in $M$
\begin{equation}
    \mathcal{M}_{J,J}^{(I)}:= \mathcal{M}_{J,J} - \mathcal{ M}_{J,I} (\mathcal{M}_{I,I})^{-1} \mathcal{M}_{I,J} . 
\end{equation}
By using the fact that $\|G -\mathcal{M}\|_{\infty} \le N^{-\varepsilon}$ and the fact that $|J|, |I| \le 2(\log N)^6$ will show after some manipulation that $\|\mathcal{M}^{(I)}_{J,J} - G^{(I)}_{J,J}\|_{\infty} \le C (\log N)^{12} N^{-\varepsilon}$. With this $\| \cdot\|_{\infty}$ bound in hand, we can show by an application of the resolvent formula that $\|(\mathcal{M}^{(I)}_{J,J})^{-1}\| - \|(G^{(I)}_{J,J})^{-1}\| = O(1)$. It suffices to understand the operator norm of $\|(\mathcal{M}^{(I)}_{J,J})^{-1}\|$ in order to understand $\|(G^{(I)}_{J,J})^{-1}\|$.

We note that it is an algebraic fact that $\mathcal{M}_{J,J}^{(I)}$ satisfies the equation,
\begin{equation}
   \mathcal{ M}_{J,J}^{(I)}= (- z - \Xi^{(I)}(\mathcal{M}^{(I)}) )^{-1}_{J,J},
\end{equation}
where $\Xi^{(I)}$ is an operator like in \eqref{eq:defXi}, but with the covariance terms $\xi$ that involve any index of $(I)$ to be set to 0. This representation allows us to determine a lower bound on the norm of $\|\mathcal{M}_{J,J}^{(I)}\|$. 

\begin{equation}
    \|\mathcal{M}_{J,J}^{(I)}\| \ge \sup_{\|v\| = 1, v \in \mathbb{R}^{|I|}}\textnormal{Im}[v^* \mathcal{M}_{J,J}^{(I)} v] \ge \sup_{v}\frac{\textnormal{Im}[v^*\Xi^{(I)}(\mathcal{M}^{(I)}_{J,J}) v]}{\|z + \Xi^{(I)}(\mathcal{M})\|^2}.
\end{equation}
The covariance structure of $\Xi^{(I)}$ and the fact that $\textnormal{Im}(\mathcal{M})$ is positive Hermitian allows one to assert that $\textnormal{Im}[v^*\Xi^{(I)}(\mathcal{M}^{(I)}_{J,J}) v]$ is bounded below. This will show that  $\|\mathcal{M}^{(I)}_{J,J}\|$ is bounded below, so $\|(\mathcal{M}^{(I)}_{J,J})^{-1}\|$ is bounded above, as desired.
\end{proof}

Now, we have the necessary estimate in order to complete our proof of the local law.
\begin{lm}
Fix some constant $\omega>0$ and recall the solution $m_{\infty}(z)$ to the self-consistent equation from \eqref{eq:sc}.
Let $\mathcal{D}$ be a subset of $\mathbb{C}^+$ such that if $z \in \mathcal{D}$ then $\textnormal{Im}[m_\infty(z)] \ge \omega$.

Fix some parameters $\nu>0$, $\sigma>0$ and $p >\frac{100}{\sigma}$. We consider the subset $\mathcal{D}_\nu: \mathcal{D} \cap \{z: \textnormal{Im}[z] \ge N^{-1 + \nu}\}$. There is some $N_{\nu,\sigma,p}$ such that for $N \ge N_{\nu,\sigma,p}$ we could derive the following  probability estimate.  
\begin{equation} \label{eq:locallaws}
\begin{aligned}
   &\mathbb{P}\left(\sup_{z \in \mathcal{D}_{\nu}} \left|\frac{1}{2N} \mathrm{Tr} G(z) - { m_{\infty}(z)}\right| \ge\frac{N^\sigma}{\sqrt{N \eta}}\right) \le N^{-\sigma p},\\
  & \mathbb{P}\left(\sup_{z \in \mathcal{D}_{\nu}, i,j}|G_{ij} -\mathcal{ M}_{ij}| \ge\frac{N^{\sigma}}{\sqrt{N \eta}}\right) \le N^{-\sigma p},\\
\end{aligned}   
\end{equation}
\end{lm}
\begin{proof}
As mentioned before, we show that it suffices to prove the probability bounds on a sufficiently dense gird on the set $\mathcal{D}_{\nu}$. First observe that the Lipschitz constants of the function $G_{ij}$ and the corresponding values in $\mathcal{M}$ are bounded by $\frac{1}{\eta^2}\le N^2$. We see that if we prove the high probability bounds on a grid whose grid distance is $N^{-4}$, we would be able to derive high probability bounds uniformly over $\mathcal{D}_{\nu}$. 

Now, we turn to establishing the result over a finite grid $\mathcal{G}$ in $\mathcal{D}_{\nu}$. Let $z$ be a grid point in $\mathcal{G}$ and let $z_1, z_2, \ldots, z_k$ be the grid points in $\mathcal{G}$ whose real part matches that of $z$. 
The points are ordered so that $\textnormal{Im}[z_1] \ge \textnormal{Im}[z_2] \ge \ldots \ge \textnormal{Im}[z_k]$. First assume that we have established the high probability bounds as in \eqref{eq:locallaws} for some $z_i$. We will now prove the results for $z_{i+1}$. 

Since we know the local law at $z_{i}$, we know that
\begin{equation}
     \|\mathcal{M}(z_k) - G(z_k)\|_{\infty} \le \frac{N^{\sigma}}{\sqrt{N \eta}},
\end{equation}
with high probability.

By using the Lipschitz continuity of $\mathcal{M}$ and $G$ in $z$. This will also establish that
\begin{equation}
    \| \mathcal{M}(z_{k+1}) - G(z_{k+1})\|_{\infty} \le N^{-\varepsilon},
\end{equation}
for some $\varepsilon >0$. This is good enough to apply Lemma \ref{lem:BoundGamma} as well as the self-consistent estimates in Lemmas \ref{lem:Stab} and \ref{lem:Self-Consist2}. This shows that we can get the desired high probaiblity bound at the point $z_{k+1}$. Taking a union bound over all elements will give us a high probability bound over our grid $\mathcal{G}$ and establish the local law.

\end{proof}

\section{ The Comparison to the Gaussian} \label{sec:OU}

A very powerful tool in proving universality of various random matrix models is the study of the Dyson Brownian motion. The study was initially pioneered in a series of papers by Erdos, Schlein, Yau, and collaborators \cite{Yau10, Erdos2009b, ERetal} and culminating in an optimal time proof of universality in \cite{ESY2}. The study of the Dyson Brownian motion has since been used to great effect in many papers, such as \cite{Erdos12c, Erdos2012,Erdos2010universality, Erdos2012b,LSY}. In this section, we apply the Dyson-Brownian motion to prove universality for the ensemble $\mathcal{H}$.

\subsection{Local Law estimates under Interpolation}

We consider the evolution of the Green's function under the modified Ornstein-Uhlenbeck (OU) process given as follows.
Recall {$H=[Y_{ij}]$}, our $2N \times 2N$ Hermitian block matrix with the $N \times N$ diagonal blocks set to $0$, and let $dB_{ij}$ be a matrix valued Brownian motion with correlation structure given by
\begin{equation}
    \textnormal{Cov}[B_{ab}(t) \overline{B}_{cd}(t)] = t \xi_{abcd}.
\end{equation}
Under this convention, we see that $B_{ab} =0, B_{N+a,N+b}$ for $1\le a,b \le N$.

We consider the matrix evolution on $H$ to be given by.
\begin{equation}
    dH_{ab}(t)= \frac{1}{\sqrt{N}} d B_{ab} - \frac{1}{2} H_{ab}(t),
\end{equation}
with $H(0)=H$ being our initial matrix and $H(t)$ be the result after running the Brownian motion for time $t$.

By our choice of Brownian motion, we see that $H(t)$ has the same covariance and independence structure as the matrix $H(0)$. Thus, we can show that a local law holds for $H(t)$ without much difficulty. 

The following integration by parts lemma will be useful in understanding the time evolution of functions of the matrix $H(t)$.
\begin{lm} \label{lem:IntbyParts}
Let $(x_1,\ldots,x_N)$ be an array of $J$ correlated random variables with mean 0 (where $J$ is allowed to be a function of $N$). Assume further that $\mathbb{E}[|x_k|^3] $ is bounded uniformly for all $k$. Pick some index $i \in [1,\ldots,N]$ and let $\mathcal{T}$ be the set of indices that are correlated with $i$. Then, we have the following relation,
\begin{equation}
    \mathbb{E}[f(x_1,\ldots,x_N) x_i]= \sum_{j \in \mathcal{T}} \mathbb{E}[\partial_j f] \mathbb{E}[x_j x_i]  +  \textnormal{O}(|K|^2 \| D^2 f\|_{\infty}).
\end{equation}
\end{lm}
\begin{proof}
The proof is an exercise in applying the Taylor expansion. 

Let $x^{(\mathcal{T})}$ be the tuple of integers $(x_1 \mathbbm{1}(1 \not \in \mathcal{T}),\ldots ,x_N \mathbbm{1}(N \not \in 
\mathcal{T}))$.
We see that if we expand $f$ in the variables in $\mathcal{T}$, we see that we can derive the expression,
\begin{equation}
    f(x_1,\ldots,x_N)x_i =f\left(x^{(\mathcal{T})}\right) x_i + \sum_{j \in \mathcal{T}} \partial_j f\left(x^{(\mathcal{T})}\right) x_j x_i + \frac{1}{2} \sum_{k,j, \in \mathcal{T}}(1-t)\int_{0}^1 \partial_{k,j}f\left(x^{(\mathcal{T})} + t \left(x- x^{(\mathcal{T})}\right)\right) x_k x_j x_i \textnormal{d} t. 
\end{equation}
We can bound,
\begin{equation}
    \frac{1}{2} \sum_{k,j \in \mathcal{T}} \int_{0}^t\mathbb{E}\left[\partial_{k,j} f\left(x^{(\mathcal{T})} + t\left(x- x^{(\mathcal{T})}\right)\right) x_k x_j x_i \right] \textnormal{d}t \le \frac{1}{2}|\mathcal{T}|^2 \| D^2 f\|_{\infty}\mathbb{E}[ |x_k|^3]^{1/3} \mathbb{E}[|x_j|^3]^{1/3} \mathbb{E}[|x_i|^3]^{1/3}.
\end{equation}

We now let $U_j$ be the set of integers that are correlated with $j$. We can again apply the Taylor expansion to compute the expectation of 
\begin{equation}
    \mathbb{E}\left[\partial_j f\left(x^{(\mathcal{T})}\right) x_j x_i\right] = \mathbb{E}\left[\partial_j f\left(x^{(\mathcal{T} \cup U_j)}\right) x_j x_i\right] + \sum_{k \in U_j \setminus \mathcal{T}} \int_{0}^1 (1-t) \mathbb{E}\left[ \partial_{j} \partial_k f\left(x^{(\mathcal{T}\cup U_j)} + t\left(x^{(\mathcal{T})} - x^{(\mathcal{T} \cup U_j)} \right)\right) x_j x_i x_k\right] \textnormal{d} t.
\end{equation}
The second term on the right hand side above can be estimated as we have done previously.

We see that
\begin{equation}
    \mathbb{E}\left[\partial_j f\left(x^{(\mathcal{T} \cup U_j)}\right) x_j x_i\right] = \mathbb{E}\left[\partial_j f\left(x^{(\mathcal{T} \cup U_j)}\right)\right] \mathbb{E}[x_j x_i].
\end{equation}
We can now reverse the application of the Taylor expansion and write,
\begin{equation}
\mathbb{E}\left[\partial_j f(x^{(\mathcal{T} \cup U_j)})\right] = \mathbb{E}\left[\partial_j f(x)\right] - \sum_{k \in U_j \cup \mathcal{T}}\int_{0}^1 (1-t) \mathbb{E}\left[\partial_k \partial_j f \left(x^{(\mathcal{T} \cup U_j)} + t \left(x - x^{(\mathcal{T} \cup U_j)} \right)\right)\right] \textnormal{d} t.
\end{equation}
Substituting this expression inside
$\mathbb{E}\left[\partial_j f\left(x^{(\mathcal{T} \cup U)}\right)\right] \mathbb{E}[x_j x_i]$ gives us the expression $\mathbb{E}[\partial_j f(x)] \mathbb{E}[x_j x_i]$ plus an error term expression which can bounded in the same way as we have done previously. This completes the proof of the expression.
\end{proof}

We will apply the previous lemma when we compute the time evolution of functions of $H(t)$.
\begin{lm} \label{lem:TimeCompare}
Let $f$ be a function in $C^3$ from $\mathbb{C}^{2N \times 2N} \to \mathbb{C}$. Then, we have the following relation,
\begin{equation}
    \mathbb{E}[f(H(t))] - \mathbb{E}[f(H(0))] = O(t N^{1/2}(\log N)^{12} \mathbb{E}[\|D^3f\|_{\infty}] ).
\end{equation}
\end{lm}
\begin{proof}
We start with applying Ito's Lemma. We see that,
\begin{equation}
    \textnormal{d} \mathbb{E}[f(H(t))] = -\frac{1}{2}\sum_{ab} \mathbb{E}[\partial_{ab}f(H(t)) x_{ab}] + \frac{1}{2N} \sum_{ab,cd} { \mathbb{E}[\partial_{ab}\overline{\partial_{cd}}f(H(t))] } \xi_{abcd}.
\end{equation}

To evaluate the first term on the right-hand side of the above equation, we may apply Lemma \ref{lem:IntbyParts}. We see that we may derive,
\begin{equation}
    -\frac{1}{2} \sum_{ab} \sum_{cd} \mathbb{E}[\partial_{ab}{ \overline{\partial_{cd}}} f(H(t))] \xi_{abcd} + \textnormal{O}((\log N)^{12} \mathbb{E}[\|D^3 f\|_{\infty}] N^{-3/2}).
\end{equation}

Here, we used the fact that $\mathbb{E}[|H_{ij}(t)|^3] = O (N^{-3/2})$. 

In what follows, it will be useful to state exactly what derivatives we need to control in the expression $\|D^3 f\|_{\infty}$, rather than apply a supremum bound. 

Given a pair $(i,j)$, we define the set $\mathcal{T}^{i,j}$ as follows
$$\mathcal{T}^{i,j}=\{(i',j'):|i - i'| \le 4(\log N)^{12} \textnormal{ or } |j -j'|< 4(\log N)^{12}\}. $$

Essentially, if one lets $\tilde{\mathcal{T}}^{i,j}$ be the set of indices of entries that could be correlated with $H_{ij}(t)$, then $\mathcal{T}^{i,j}$ is the set of indices of entries that could be correlated with entries whose indices are in $\mathcal{T}^{i,j}$.

When we apply the Taylor expansion, we see that we consider expressions of the form.
\begin{equation}
    \partial_{ij} \sum_{ab,cd \in \mathcal{T}^{i,j}} \partial_{ab} { \overline{\partial_{cd}}} f\left(H(t)^{(\mathcal{T}^{i,j})} + \theta  \left(H(t) - H(t)^{(\mathcal{T}^{i,j})}\right)\right) .
\end{equation}

Here, we apply the notation from Lemma \ref{lem:IntbyParts} to let $H(t)^{\mathcal{T}^{i,j}}$ to represent the matrix $H(t)$ with certain entries set to $0$.

In the proof of the previous lemma, $\theta$ is a constant between $0$ and $1$.
\end{proof}

With the above lemma in hand, we can now establish a Green's function comparison theorem. 
\begin{lm} \label{thm:NewGFCT}
	Recall the setting of  Lemma \ref{lm:GFCL}; namely, let $n\geq 1$, $\eps>0$ and let $E_1,\ldots,E_n$ satisfy { $\rho_\infty(E_i)\geq \eps$, where  $\rho_{\infty}$ is the density associated with $m_{\infty}$}. Given $\sigma_1,\ldots,\sigma_n\leq \sigma$, we set
	$$
	z_j=E_j+\ti \eta_j,\qquad \eta_j=N^{-1-\sigma_j}.
	$$
Consider the matrix dynamics $H(t)$ with $H(0)$ coming from our initial matrix distribution as in \eqref{eq:Tildmat}.  We let $G^t$ be the Green's function of $H(t) -z$ with normalized trace $m^t$ and $G^0$ be the Green's function of $H(0) -z$ with normalized trace $m^0$. Then, there exists $C_\sigma>0$ depending only on $\sigma$ such that
$$
\l|\prod_{k=1}^n\mathrm{Im}\, m^t(z_k)-\prod_{k=1}^n\mathrm{Im}\, m^0(z_k)\r|\leq \frac{C_\sigma}{N}.
$$
\end{lm}
\begin{proof}
We will prove the comparison when $n=1$. The proof of the general statement follows similar details.

 We will try to apply the previous Lemma \ref{lem:TimeCompare}. We see that it suffices to derive a bound on the third derivatives,
\begin{equation}
   \mathbb{E} \left| \partial_{ab}\partial_{cd}\partial_{ef} \frac{1}{2N} \text{Tr}\left( G^{\theta,s}(z)\right) \right|.
\end{equation}
Here, $cd$ and $ef$ are entries in $\mathcal{T}^{a,b}$ and 
 $G^{\theta,s}(z)$ is the Green's function of the matrix $ H(s)- z + \theta (H^{\mathcal{T}^{a,b}}(s) -H(s))$, where $\theta$ is a constant between $0$ and $1$.
 
 First fix a time $s$ and set $\theta=0$. We will first establish a bound here before discussing the general case.
 
 By direct computation, one can see that $$
 \left|\partial^{\alpha} \frac{1}{2N} \mathrm{Tr}[G^{s}(z)] \right| \le \Gamma^4(z),
 $$
 where, recall, $\Gamma(z)$ is a uniform upper bound for the entries of $G^s(z)$ and $\partial^{\alpha}$ indicates any third order partial derivative.
 
 Again, by direct computation, one can see that the change of $\Gamma$ as the imaginary part $\eta$ of $z$ changes satisfies a useful inequality,
 $$
 \left|\frac{\partial \Gamma}{\partial \eta} \right| \le \frac{\Gamma}{\eta}.
 $$
 
 One can integrate this differential equation to see that
 \begin{equation}
     \Gamma(E + \textnormal{i} \eta) \le \Gamma(E + \textnormal{i} N^{-1+ \epsilon}) N^{2\epsilon}.
 \end{equation}
 whenever $N^{-1 -\epsilon} \le \eta \le N^{-1}$ and $\eta$. Now, when $E$ is in the bulk of the distribution, we can apply our local law to ensure that $\Gamma(E + \textnormal{i} N^{-1 + \epsilon}) = O(N^{\epsilon})$. Thus, with probability $1- N^{-D}$ for some large $D$, we could ensure that
 $\left|\partial^{|\alpha|} \frac{1}{2N} \mathrm{Tr}[G^s(z)]\right| \le C N^{12 \epsilon}$ and has the trivial bound $N^{8}$ otherwise.
 
 
 By Lipschitz continuity, one can establish these results on a discrete grid of times and extend to the entire interval $[0,t]$. In addition, one can show that matrices of the form $H(t)^{\mathcal{T}^{i,j}} + \theta \left( H(t)- H(t)^{\mathcal{T}^{i,j}}\right)$ satisfy a similar local law. Again, applying local law results to a discrete grid of $\theta$s and noting the fact that there are no more than $N^2$ choices of these special $\mathcal{T}^{i,j}$ modifications will allow us to get a uniform probability bound on all choices of $\theta$ and $i,j$. This gives us a desired proof of the bounds on the derivatives we need to apply Lemma \ref{lem:TimeCompare} and complete the proof of the Theorem.

\end{proof}

The results of the above Green's function comparison theorem can be used to prove the following comparison on correlation functions, as we have seen earlier in the proof of Lemma \ref{lm:CFCL}.

\begin{thm} \label{thm:CorrelationFunctionCompar}
Fix a time $t= N^{-1 + \epsilon}$ with $\epsilon>0$. Consider the matrix dynamics { $H(t)$ with $H(0)$} coming from our initial matrix distribution as in \eqref{eq:Tildmat}. Let $p_{N}^{(k),t}$ be the correlation functions of $H(t)$. Let $\rho$ be the density corresponding the limiting spectral distribution of $H(0)$. and let $E$ be a point in the support of $\rho$. Then, for any compactly supported continuous test function $O$ from $\mathbb{R}^k \to \mathbb{R}$, we have the following comparison estimate,
\begin{equation}
    \int_{\mathbb{R}^k} O(\bm{\al}) \left[p_{N}^{(k),t}\left(E + \frac{\bm{\al}}{2N}\right) -p_{N}^{(k),0}\left(E+  \frac{{\bm{\al}}}{2N}\right) \right] \textnormal{d}{\bm{\al}} = \textnormal{O}(N^{-c})
\end{equation}
\end{thm}
\subsection{Comparing to the GOE}

At this point, we have established that the statistics of $H(0)$ match those of $H(t)$. We will be finished once we show that the statistics of $H(t)$ match those of the GOE. 

However, recall from our interpolation that $H(t)$ has a correlated Gaussian component $\sqrt{t}{N} \begin{bmatrix} &0 &C\\
&C^* &0\end{bmatrix}$, where $G$ has the correlation structure given by $\mathbb{E}[C_{ab} C_{cd}] = \xi_{abcd}$. However, Because our covariance matrix $\xi< c_0$ is positive semidefinite, we can split the matrix $C$ as $C = \tilde{C} + G$ where $\tilde{C}$ and $G$ are independent Gaussians and $G$ is a { GUE }matrix.

Thus, the matrix $H(t)$ can be represented in the form $\tilde{H} + c 
{ GUE}$ for $c \ge N^{-1 + \epsilon}$ and $\tilde{H}$ independent of the { GUE}. Theorem 2.2 of \cite{LSY} proves that the matrix $H(t)$ will have universal spectral statistics. Theorem \ref{thm:CorrelationFunctionCompar} shows that $H(0)$ will have the same spectral statistics as $H(t)$. This proves Theorem \ref{thm:auxmain}. Finally using Lemma \ref{lm:CFCL}, this will further prove Theorem \ref{thm:main}.

\begin{appendix}
\section{On the limiting objects $m_\infty(z)$ and $\rho_\infty(z)$}\label{app:limiting}
In this appendix, we prove Propositions \ref{prop:minfty} and \ref{prop:rhoinfty}.

\be{proof}[Proof of Proposition \ref{prop:minfty}]
This was essentially already proved in the main text in the pre-limit case. Existence of $m_\infty$ follows from Brouwer's fixed point theorem as in the proof of Lemma \ref{lem:exist}. Uniqueness is proved by following the argument establishing Lemma \ref{lem:uniq}. We omit the details.
\e{proof}

\be{proof}[Proof of Proposition \ref{prop:rhoinfty}]
In a first step, we prove that $m_\infty(z)$ is the Stieltjes transform of a Borel measure on $\R$. The analytic implicit function theorem and the condition on the imaginary part imply that $m_\infty(z)$ is a Herglotz function. Hence, the Herglotz representation theorem yields constants $a\in\R$, $b\geq0$, and a Borel measure $\d\rho_\infty$ on $\R$ satisfying $\int_\R \tfrac{1}{1+x^2}\d\rho_\infty(x)<\infty$ such that
$$
m_\infty(z)=a+bz+\int_\R \l(\frac{1}{x-z}-\frac{1}{1+x^2}\r)\d\rho_\infty(x)
$$
Let $G$ be an $N\times N$ Gaussian matrix with the same correlation structure as the matrix $Y$ from \eqref{eq:Ydefn} and define its Hermitization
$$
H_G=\be{pmatrix}
 0 & G\\
 G^\dagger & 0
\e{pmatrix}
$$
We can repeat the proof of the local law that was given for $\tilde H$ for the matrix $H_G$ because only the correlation structure and range of dependence (which equals the range of correlation for a Gaussian matrix) is used. The upshot is that $m_\infty(z)$ arises as the limiting spectral density of the Gaussian matrix ensemble $H_G$. Since $H_G$ are Hermitian matrices, this implies
$$
|m_\infty(z)|\leq \frac{1}{\mathrm{Im}\, z},\qquad \mathrm{Im}\,z>0.
$$
By considering the asymptotics of this estimate for $z=iy$ with $y\to\infty$, the Herglotz representation formula reduces to
\beq\label{eq:minftyzintermediate} 
m_\infty(z)=\int_\R \frac{1}{x-z}\d\rho_\infty(x),
\eeq
for a finite Borel measure $\d\rho_\infty(x)$.

In a second step, we use free probability theory to prove that the Borel measure $\d\rho_\infty(x)$ in \eqref{eq:minftyzintermediate}  has a continuous density. 

Recall that $f$ is an admissible evaluation function. By Lemma \ref{lm:PhiNest}, this implies that the correlation matrix $\Phi^N\geq \frac{g_{\min}}{2}>0$ is strictly positive definite. Hence, we can decompose 
$$
G=\frac{g_{\min}}{2}G_1+G_2
$$
where $G_1$ is a Ginibre matrix (independent Gaussian entries with variance $\frac{1}{N}$) and $G_2 $ is a Gaussian matrix independent of $G_1$ with entries of variance $\sim\frac{1}{N}$ and correlation matrix $\geq\frac{g_{\min}}{2}$. This decomposition extends to the Hermitization
$$
H_G=\frac{g_{\min}}{2}H_1+H_2,\qquad H_i=\be{pmatrix}
 0 & G_i\\
 G_i^\dagger & 0
\e{pmatrix}.
$$
Both matrices $H_1$ and $H_2$ satisfy a local law for respective limiting densities $\d\rho_1(x)\equiv \d\rho_{\mathrm{sc}}(x)$ and $\d\rho_2(x)$. Thus $\d\rho_\infty(x)$ arises as the free convolution of $\d\rho_2$ with the Wigner semicircle law. A result of Biane \cite[Cor.\ 2]{Biane} then says that $\d\rho_\infty(x)$ has a continuous density. It is computable from \eqref{eq:minftyzintermediate}  via the Stieltjes inversion formula.
\e{proof}

\section{Heuristic derivation of the self-consistent equation}\label{app:heuristic}

Recall, by definition, we have the equation $G Y - z G = I$. To derive the form of the self consistent equation, it would be required to consider a matrix $\tilde{Y}$ of Gaussian random variables whose covariance structure matches the covariance structure of $Y$; thus, we consider the equation $G \tilde{Y} - z G = I$, take the expectation of both sides of the equation, and simplify by integrating by parts with respect to the Gaussian variables in $\tilde{Y}$. This  procedure will result in the equation
\begin{equation}
    \mathbb{E}[(G \tilde{Y} -z G)_{ab}] = \mathbb{E}[-\frac{1}{N} \sum_{k,l,m}G_{ik} \xi_{kl jm}G_{lm} - z G_{ij}] =\delta_{ij}.
\end{equation}
Here, we wrote $(G\tilde{Y})_{ij} = \sum G_{im} \tilde{Y}_{mj}$. Observe now that $\partial_{\tilde{Y}_{kl}} G_{im} = G_{ik} G_{lm}$. Integrating by parts, using the fact that $\mathbb{E}[\tilde{Y}_{kl}\tilde{Y}_{mj}] = \xi_{kl jm}$ will show that there is a prefactor of $\xi_{kl jm} G_{ik} G_{lm}$ associated with this quantity. At the last step we can remove the expectation, anticipating that these quantities will be concentrated.

\end{appendix}

\section*{Acknowledgments}
A.A. would like to thank NSF award number 2102842  and the Harvard GSAS Merit/Graduate Society Term-Time Research Fellowship for support during part of this project
\bibliographystyle{plain}

\bibliography{mixing}

\begin{thebibliography}{10}

\bibitem{AC}
A.~Adhikari and Z.~Che.
\newblock Edge universality for correlated gaussians.
\newblock {\em Elec. Jour. Prob.}, 24, 2019.

\bibitem{AL}
A.~Adhikari and M.~Lemm.
\newblock A local law for singular values from diophantine equations.
\newblock {\em International Mathematics Research Notices (to appear), arXiv
  preprint arXiv:2005.04102}, 2020.

\bibitem{ALY}
A~Adhikari, M.~Lemm, and H.T. Yau.
\newblock Global eigenvalue distribution of matrices defined by the skew-shift.
\newblock {\em Anal. PDE}, 14:1153--1198, 2021.

\bibitem{Erdos2015}
O.~Ajanki, L.~Erd\H{o}s, and T.~Kruger.
\newblock Local spectral statistics of gaussian matrices with correlated
  entries.
\newblock {\em Journal of Statistical Physics}, pages 1--23, 2016.

\bibitem{AEKS}
J.~Alt, L.~Erdos, T.~Kruger, and D.~Schroder.
\newblock Correlated random matrices:band rigidity and edge universality.
\newblock {\em Ann. Prob.}, 48:963--1001, 2020.

\bibitem{AJ}
A.~Avila and S.~Jitomirskaya.
\newblock The ten martini problem.
\newblock {\em Annals of mathematics}, pages 303--342, 2009.

\bibitem{ADZ}
Artur Avila, David Damanik, and Zhenghe Zhang.
\newblock Schr$\backslash$" odinger operators with potentials generated by
  hyperbolic transformations: I. positivity of the lyapunov exponent.
\newblock {\em arXiv preprint arXiv:2011.10146}, 2020.

\bibitem{Banna13}
M.~Banna, F.~Merlev\'{e}de, and M.~Peligrad.
\newblock On the limiting spectral distribution for a large class of random
  matrices with correlated entries.
\newblock {\em Stochastic Processes and their Applications}, 125:2700--2726,
  2015.

\bibitem{BHKY}
R~Bauerschmidt, J~Huang, A~Knowles, and H.T. Yau.
\newblock Bulk eigenvalue statistics for random regular graphs.
\newblock {\em Ann. Probab.}, 45:3626--3663, 2017.

\bibitem{Bau15}
R.~Bauerschmidt, A.~Knowles, and H-T Yau.
\newblock Local semicircle law for random regular graphs.
\newblock {\em Communications on Pure and Applied Mathematics}, 70:1898--1960,
  2017.

\bibitem{Biane}
P.~Biane.
\newblock On the free convolution with a semi-circular distribution.
\newblock {\em Indiana University Mathematics Journal}, pages 705--718, 1997.

\bibitem{BGS01}
J~Bourgain, M~Goldstein, and W~Schlag.
\newblock Anderson localization for schrodinger operators on z with potentials
  given by skew-shift.
\newblock {\em Comm. Math. Phys.}, 220:583--621, 2001.

\bibitem{BS}
J.~Bourgain and W.~Schlag.
\newblock Anderson localization for schr{\"o}dinger operators on z with
  strongly mixing potentials.
\newblock {\em Communications in Mathematical Physics}, 215(1):143--175, 2000.

\bibitem{Bowen}
R.~Bowen.
\newblock Equilibrium states and the ergodic theory of anosov diffeomorphisms.
\newblock {\em Springer Lecture Notes in Math}, 470:78--104, 1975.

\bibitem{Che2016}
Z.~Che.
\newblock Universality of random matrices with correlated entries.
\newblock {\em Elec. Jour. Probab.}, 22, 2017.

\bibitem{CS}
T.~Chulaevsky, V.and~Spencer.
\newblock Positive lyapunov exponents for a class of deterministic potentials.
\newblock {\em Communications in mathematical physics}, 168(3):455--466, 1995.

\bibitem{DKS}
D.~Damanik, R.~Killip, and B.~Simon.
\newblock Perturbations of orthogonal polynomials with periodic recursion
  coefficients.
\newblock {\em Annals of mathematics}, pages 1931--2010, 2010.

\bibitem{Boutet96}
B.~de~Monvel, A.~Khorunzhy, and V.~Vasilchuk.
\newblock Limiting eigenvalue distribution of random matrices with correlated
  entries.
\newblock {\em Markov Process. Related Fields}, 2:607--636, 1996.

\bibitem{Erdos2009b}
L.~Erd\H{o}s, S.~Peche, J.~Ramirez, B.~Schlein, and H-T. Yau.
\newblock {Bulk Universality for Wigner Matrices}.
\newblock {\em Communications on Pure and Applied Mathematics}, 63:895--925,
  2010.

\bibitem{Yau10}
L.~Erd\H{o}s, J.~Ramirez, B.~Schlein, and H-T. Yau.
\newblock Universality of sine-kernel for wigner matrices with a small gaussian
  perturbation.
\newblock {\em Electronic Journal of Probability}, 15:526--604, 2010.

\bibitem{Erdos12c}
L.~Erd\H{o}s, B.~Schlein, H-T. Yau, and J.~Yin.
\newblock {The local relaxation flow approach to universality of the local
  statistics for random matrices}.
\newblock {\em Annales de l'institut Henri Poincare (B) Probability and
  Statistics}, 48(1):1--46, 2012.

\bibitem{Erdos2012}
L.~Erd\H{o}s, H-T. Yau, and J.~Yin.
\newblock {Bulk universality for generalized Wigner matrices}.
\newblock {\em Probability Theory and Related Fields}, 154(1-2):341--407, 2012.

\bibitem{Erdos2012b}
L.~Erd\H{o}s, H-T. Yau, and J.~Yin.
\newblock {Rigidity of eigenvalues of generalized Wigner matrices}.
\newblock {\em Adv. Math. (N. Y).}, 229(3):1435--1515, 2012.

\bibitem{EKS}
L~Erdos, T.~Kruger, and D.~Schroder.
\newblock Random matrices with slow correlation decay.
\newblock {\em Forum of Mathematics Sigma}, 7, 2019.

\bibitem{ERetal}
L.~Erdos, J.~Ramirez, B.~Schlein, T.~Tao, V.~Vu, and H.T. Yau.
\newblock Bulk universality for wigner hermitian matrices with subexponential
  decay.
\newblock {\em Math. Res. Lett.}, 17:667--67, 2010.

\bibitem{ESY2}
L.~Erdos, B.~Schlein, and H.T. Yau.
\newblock Universality of random matrices and local relaxation flow.
\newblock {\em Invent. Math.}, 185:75--119, 2011.

\bibitem{bYau}
L.~Erdos and H-T. Yau.
\newblock A dynamical approach to random matrix theory.
\newblock {\em Courant Lecture Notes in Mathematics}, 28, 2017.

\bibitem{Erdos2010universality}
L.~Erd{\H{o}}s, H-T. Yau, and J.~Yin.
\newblock Universality for generalized {W}igner matrices with {B}ernoulli
  distribution.
\newblock {\em J. of Combinatorics}, 2:15--85, 2011.

\bibitem{GrayToeplitz}
R~Gray.
\newblock Toeplitz and circulant matrices: A review.
\newblock {\em Foundations and Trends in Communications and Information
  Theory}, 2:155--239, 2006.

\bibitem{HLS1}
R~Han, M~Lemm, and W~Schlag.
\newblock Effective multiscale approach to the schrodinger cocycle over a
  skew-shift base.
\newblock {\em Ergod. Theory Dyn. Syst.}, 40, 2020.

\bibitem{HLa}
J.~Huang and B.~Landon.
\newblock Spectral statistics of sparse erdos-renyi graph laplacians.
\newblock {\em Ann. Henri Poincare Prob et Stat}, 56, 2020.

\bibitem{JL}
S.~Jitomirskaya and W.~Liu.
\newblock Universal hierarchical structure of quasiperiodic eigenfunctions.
\newblock {\em Annals of Mathematics}, 187(3):721--776, 2018.

\bibitem{Khorunzhy94}
A.~Khorunzhy.
\newblock Eigenvalue distribution of large random matrices with correlated
  entries.
\newblock {\em Matematicheskaya fizika, analiz, geometriya}, 3(1/2):80--101,
  1996.

\bibitem{K1}
H~Kruger.
\newblock Multiscale analysis for ergodic schrodinger operators and positivity
  of lyapunov exponents.
\newblock {\em J. Anal. Math.}, 115:343--387, 2011.

\bibitem{LSY}
B.~Landon, P.~Sosoe, and H-T. Yau.
\newblock Fixed energy universality of dyson brownian motion.
\newblock {\em Advances in Mathematics}, 346:1137--1332, 2019.

\bibitem{Minami}
N.~Minami.
\newblock Local fluctuation of the spectrum of a multidimensional anderson
  tight binding model.
\newblock {\em Communications in mathematical physics}, 177(3):709--725, 1996.

\bibitem{Tao2011a}
T.~Tao and V.~Vu.
\newblock {Random matrices: Universality of local eigenvalue statistics}.
\newblock {\em Acta Math.}, 206(1):127--204, 2011.

\end{thebibliography}


\end{document}